\documentclass[12pt,oneside,english]{amsart}
\usepackage{lmodern}
\usepackage[T1]{fontenc}
\usepackage[latin9]{inputenc}
\usepackage{geometry}
\usepackage{color}
\usepackage{babel}
\usepackage{verbatim}
\usepackage{amsbsy}
\usepackage{amstext}
\usepackage{amsthm}
\usepackage{amssymb}
\usepackage[numbers]{natbib}
\usepackage[unicode=true,pdfusetitle,
 bookmarks=true,bookmarksnumbered=false,bookmarksopen=false,
 breaklinks=false,pdfborder={0 0 1},backref=false,colorlinks=false]
 {hyperref}

\makeatletter
\numberwithin{equation}{section}
\numberwithin{figure}{section}
\theoremstyle{definition}
\newtheorem{defn}{\protect\definitionname}
\theoremstyle{remark}
\newtheorem*{notation*}{\protect\notationname}
\theoremstyle{plain}
\newtheorem{assumption}{\protect\assumptionname}
\theoremstyle{remark}
\newtheorem*{rem*}{\protect\remarkname}
\theoremstyle{plain}
\newtheorem{thm}{\protect\theoremname}
\theoremstyle{plain}
\newtheorem{lem}{\protect\lemmaname}
\theoremstyle{plain}
\newtheorem{prop}{\protect\propositionname}
\theoremstyle{remark}
\newtheorem*{acknowledgement*}{\protect\acknowledgementname}

\usepackage{enumitem}
\setlist{nosep}
\usepackage{graphicx}
\usepackage{pgf,tikz}
\tikzset{every picture/.style={/utils/exec={\sffamily}}}
\usetikzlibrary{arrows}

\makeatother

\providecommand{\acknowledgementname}{Acknowledgement}
\providecommand{\assumptionname}{Assumption}
\providecommand{\definitionname}{Definition}
\providecommand{\lemmaname}{Lemma}
\providecommand{\notationname}{Notation}
\providecommand{\propositionname}{Proposition}
\providecommand{\remarkname}{Remark}
\providecommand{\theoremname}{Theorem}

\begin{document}
\newcommand{\R}{\mathbb{R}}
\title{Nonlocal Equations with Gradient Constraints}
\author{Mohammad Safdari$\,{}^{1}$}
\begin{abstract}
We prove the existence and $C^{1,\alpha}$ regularity of solutions
to nonlocal fully nonlinear elliptic equations with gradient constraints.
We do not assume any regularity about the constraints; so the constraints
need not be $C^{1}$ or strictly convex. We also obtain $C^{0,1}$
boundary regularity for these problems. Our approach is to show that
these nonlocal equations with gradient constraints are related to
some nonlocal double obstacle problems. Then we prove the regularity
of the double obstacle problems. In this process, we also employ the
monotonicity property for the second derivative of obstacles, which
we have obtained in a previous work.\medskip{}

\thanks{$^{1}\;$Department of Mathematical Sciences, Sharif University of
Technology, Tehran, Iran\protect \\
Email address: safdari@sharif.edu}
\end{abstract}

\maketitle

\section{Introduction}

In this paper we consider the existence and regularity of solutions
to the equation with gradient constraint 
\begin{equation}
\begin{cases}
\max\{-Iu,\;H(Du)\}=0 & \textrm{in }U,\\
u=\varphi & \textrm{in }\mathbb{R}^{n}-U.
\end{cases}\label{eq: PDE grad constr}
\end{equation}
Here $I$ is a nonlocal elliptic operator, of which a prototypical
example is the fractional Laplacian 
\[
-(-\Delta)^{s}u(x)=c_{n,s}\int_{\mathbb{R}^{n}}\frac{u(x+y)+u(x-y)-2u(x)}{|y|^{n+2s}}\,dy.
\]
Nonlocal operators appear naturally in the study of discontinuous
stochastic processes as the jump part of their infinitesimal generator.
These operators have also been studied extensively in recent years
from the analytic viewpoint of integro-differential equations. The
foundational works of \citet{caffarelli2009regularity,caffarelli2011evans,caffarelli2011regularity}
paved the way and set the framework for such studies. They provided
an appropriate notion of ellipticity for nonlinear nonlocal equations,
and obtained their $C^{1,\alpha}$ regularity. They also obtained
Evans-Krylov-type $C^{2s+\alpha}$ regularity for convex equations.
An interesting property of their estimates is their uniformity as
$s\uparrow1$, which provides a new proof for the corresponding classical
estimates for local equations.

Free boundary problems involving nonlocal operators have also seen
many advancements. \citet{silvestre2007regularity} obtained $C^{1,\alpha}$
regularity of the obstacle problem for fractional Laplacian. \citet{caffarelli2008regularity}
proved the optimal $C^{1,s}$ regularity for this problem when the
obstacle is smooth enough. \citet*{bjorland2012nonlocal} studied
a double obstacle problem for the infinity fractional Laplacian which
appear in the study of a nonlocal version of the tug-of-war game.
\citet{korvenpaa2016obstacle} studied the obstacle problem for operators
of fractional $p$-Laplacian type. \citet{petrosyan2015optimal} considered
the obstacle problem for the fractional Laplacian with drift in the
subcritical regime $s\in(\frac{1}{2},1)$, and \citet{fernandez2018obstacle}
studied the critical case $s=\frac{1}{2}$. There has also been some
works on other types of nonlocal free boundary problems, like the
work of \citet{rodrigues2019nonlocal} on nonlocal linear variational
inequalities with constraint on the fractional gradient.

A major breakthrough in the study of nonlocal free boundary problems
came with the work of \citet{caffarelli2017obstacle}, in which they
obtained the regularity of the solution and of the free boundary of
the obstacle problem for a large class of nonlocal elliptic operators.
These problems appear naturally when considering optimal stopping
problems for Lévy processes with jumps, which arise for example as
option pricing models in mathematical finance. We should mention that
in their work, the boundary regularity results of \citet{ros2016boundary,ros2017boundary}
for nonlocal elliptic equations were also essential. In \citep{Safd-Nonlocal-dbl-obst}
we proved existence and $C^{1,\alpha}$ regularity of solutions to
double obstacle problems for a wide class of nonlocal fully nonlinear
operators. We also obtained their boundary regularity. In contrast
to \citep{caffarelli2017obstacle}, we allowed less smooth obstacles,
and did not require them to be $C^{1}$.

In this paper we prove $C^{1,\alpha}$ regularity of the equation
with gradient constraint (\ref{eq: PDE grad constr}) for a large
class of nonlocal fully nonlinear operators $I$. We do not require
the operator to be convex. We also do not require the constraint to
be strictly convex or differentiable. Furthermore, we obtain $C^{0,1}$
boundary regularity for these problems, which is more than the expected
boundary regularity for nonlocal Dirichlet problems, due to the presence
of the constraint. Our estimates in this work are uniform as $s\uparrow1$;
hence they provide a new proof for the corresponding regularity results
for local equations with gradient constraints. Nonlocal equations
with gradient constraints appear for example in portfolio optimization
with transaction costs when prices are governed by Lévy processes
with jumps; see \citep{benth2001optimal,benth2002portfolio} and {[}\citealp{oksendal2007applied},
Chapter 8{]} for more details. \textcolor{red}{}

Let us now mention some of the works on local equations with gradient
constraints. The study of elliptic equations with gradient constraints
was initiated by \citet{MR529814} when he considered the problem
\[
\max\{Lu-f,\;|Du|-g\}=0,
\]
where $L$ is a (local) linear elliptic operator of the form 
\[
Lu=-a_{ij}D_{ij}^{2}u+b_{i}D_{i}u+cu.
\]
Equations of this type stem from dynamic programming in a wide class
of singular stochastic control problems. Evans proved $W_{\mathrm{loc}}^{2,p}$
regularity for $u$. He also obtained the optimal $W_{\mathrm{loc}}^{2,\infty}$
regularity under the additional assumption that $a_{ij}$ are constant.
\citet{MR607553} removed this additional assumption and obtained
$W_{\mathrm{loc}}^{2,\infty}$ regularity in general. Later, \citet{MR693645}
allowed the gradient constraint to be more general, and proved global
$W^{2,\infty}$ regularity. We also mention that \citet{soner1989regularity,soner1991free}
considered similar problems with special structure, and proved the
existence of classical solutions.

\citet{yamada1988hamilton} allowed the differential operator to be
more general, and proved the existence of a solution in $W_{\mathrm{loc}}^{2,\infty}$
to the problem 
\[
\max_{1\le k\le N}\{L_{k}u-f_{k},\;|Du|-g\}=0,
\]
where each $L_{k}$ is a (local) linear elliptic operator. Recently,
there has been new interest in these types of problems. \citet{hynd2013analysis}
considered problems with more general gradient constraints of the
form 
\[
\max\{Lu-f,\;\tilde{H}(Du)\}=0,
\]
where $\tilde{H}$ is a convex function. He proved $W_{\mathrm{loc}}^{2,\infty}$
regularity when $\tilde{H}$ is strictly convex. Finally, \citet{Hynd}
studied (local) fully nonlinear elliptic equations with strictly convex
gradient constraints of the form 
\[
\max\{F(x,D^{2}u)-f,\;\tilde{H}(Du)\}=0.
\]
Here $F(x,D^{2}u)$ is a fully nonlinear elliptic operator. They obtained
$W_{\mathrm{loc}}^{2,p}$ regularity in general, and $W_{\mathrm{loc}}^{2,\infty}$
regularity when $F$ does not depend on $x$. Let us also mention
that \citet{Hynd-2012,Hynd2017} considered eigenvalue problems for
(local) equations with gradient constraints too.

Closely related to these problems are variational problems with gradient
constraints. An important example among them is the well-known elastic-plastic
torsion problem, which is the problem of minimizing the functional
$\int_{U}\frac{1}{2}|Dv|^{2}-v\,dx$ over the set 
\[
W_{B_{1}}:=\{v\in W_{0}^{1,2}(U):|Dv|\le1\textrm{ a.e.}\}.
\]
The $W^{2,p}$ regularity for this problem was proved by \citet{MR0239302},
and its optimal $W_{\mathrm{loc}}^{2,\infty}$ regularity was obtained
by \citet{MR513957}. An interesting property of variational problems
with gradient constraints is that under mild conditions they are equivalent
to double obstacle problems. For example the minimizer of $\int_{U}G(Dv)\,dx$
over $W_{B_{1}}$ also satisfies $-d\le v\le d$ and 
\[
\begin{cases}
-D_{i}(D_{i}G(Dv))=0 & \textrm{ in }\{-d<v<d\},\\
-D_{i}(D_{i}G(Dv))\le0 & \textrm{ a.e. on }\{v=d\},\\
-D_{i}(D_{i}G(Dv))\ge0 & \textrm{ a.e. on }\{v=-d\},
\end{cases}
\]
where $d$ is the distance to $\partial U$; see for example \citep{MR1,SAFDARI202176}.
This problem can be more compactly written as 
\[
\max\{\min\{F(x,D^{2}v),v+d\},v-d\}=0,
\]
where $F(x,D^{2}v)=-D_{i}(D_{i}G(Dv))=-D_{ij}^{2}G(x)D_{ij}^{2}v$.

Variational problems with gradient constraints have also seen new
developments in recent years. \citet{MR2605868} investigated the
minimizers of some functionals subject to gradient constraints, arising
in the study of random surfaces. In their work, the functional is
allowed to have certain kinds of singularities. Also, the constraints
are given by convex polygons; so they are not strictly convex. They
showed that in two dimensions, the minimizer is $C^{1}$ away from
the obstacles. In {[}\citealp{Safdari20151}\nocite{safdari2017shape,MR1}\textendash \citealp{MR1},
\citealp{SAFDARI202176}{]} we have studied the regularity and the
free boundary of several classes of variational problems with gradient
constraints. Our goal was to understand the behavior of these problems
when the constraint is not strictly convex; and we have been able
to obtain their optimal $C^{1,1}$ regularity in arbitrary dimensions.
This has been partly motivated by the above-mentioned problem about
random surfaces. 

There has also been similar interests in elliptic equations with gradient
constraints which are not strictly convex. These problems emerge in
the study of some singular stochastic control problems appearing in
financial models with transaction costs; see for example \citep{barles1998option,possamai2015homogenization}.
In \citep{SAFDARI2021358} we extended the results of \citep{Hynd}
and proved the optimal $C^{1,1}$ regularity for (local) fully nonlinear
elliptic equations with non-strictly convex gradient constraints.
Our approach was to obtain a link between double obstacle problems
and elliptic equations with gradient constraints. This link has been
well known in the case where the double obstacle problem reduces to
an obstacle problem. However, we have shown that there is still a
connection between the two problems in the general case. In this approach,
we also studied (local) fully nonlinear double obstacle problems with
singular obstacles. 

Now let us introduce the problem in more detail. First we recall some
of the definitions and conventions about nonlocal operators introduced
in \citep{caffarelli2009regularity}. Let 
\[
\delta u(x,y):=u(x+y)+u(x-y)-2u(x).
\]
A linear nonlocal operator is an operator of the form 
\[
Lu(x)=\int_{\mathbb{R}^{n}}\delta u(x,y)a(y)\,dy,
\]
where the kernel $a$ is a positive function which satisfies $a(-y)=a(y)$,
and 
\[
\int_{\mathbb{R}^{n}}\min\{1,|y|^{2}\}a(y)\,dy<\infty.
\]

We say a function $u$ belongs to $C^{1,1}(x)$ if there are quadratic
polynomials $P,Q$ such that $P(x)=u(x)=Q(x)$, and $P\le u\le Q$
on a neighborhood of $x$. A nonlocal operator $I$ is an operator
for which $Iu(x)$ is well-defined for bounded functions $u\in C^{1,1}(x)$,
and $Iu(\cdot)$ is a continuous function on an open set if $u$ is
$C^{2}$ over that open set. The operator $I$ is \textit{uniformly
elliptic} with respect to a family of linear operators $\mathcal{L}$
if for any bounded functions $u,v\in C^{1,1}(x)$ we have 
\begin{equation}
M_{\mathcal{L}}^{-}(u-v)(x)\le Iu(x)-Iv(x)\le M_{\mathcal{L}}^{+}(u-v)(x),\label{eq: I elliptic}
\end{equation}
where the extremal Pucci-type operators $M_{\mathcal{L}}^{\pm}$
are defined as 
\[
M_{\mathcal{L}}^{-}u(x)=\inf_{L\in\mathcal{L}}Lu(x),\qquad\qquad M_{\mathcal{L}}^{+}u(x)=\sup_{L\in\mathcal{L}}Lu(x).
\]
Let us also note that $\pm M_{\mathcal{L}}^{\pm}$ are subadditive
and positively homogeneous. 

An important family of linear operators is the class $\mathcal{L}_{0}$
of linear operators whose kernels are comparable with the kernel of
fractional Laplacian $-(-\Delta)^{s}$, i.e. 
\begin{equation}
(1-s)\frac{\lambda}{|y|^{n+2s}}\le a(y)\le(1-s)\frac{\Lambda}{|y|^{n+2s}},\label{eq: L_0}
\end{equation}
where $0<s<1$ and $0<\lambda\le\Lambda$. It can be shown that in
this case the extremal operators $M_{\mathcal{L}_{0}}^{\pm}$, which
we will simply denote by $M^{\pm}$, are given by 
\begin{align}
 & M^{+}u=(1-s)\int_{\mathbb{R}^{n}}\frac{\Lambda\delta u(x,y)^{+}-\lambda\delta u(x,y)^{-}}{|y|^{n+2s}}\,dy,\nonumber \\
 & M^{-}u=(1-s)\int_{\mathbb{R}^{n}}\frac{\lambda\delta u(x,y)^{+}-\Lambda\delta u(x,y)^{-}}{|y|^{n+2s}}\,dy,\label{eq: M+-}
\end{align}
where $r^{\pm}=\max\{\pm r,0\}$ for a real number $r$. 

We will also only consider ``constant coefficient'' nonlocal operators,
i.e. we assume that $I$ is translation invariant: 
\[
I(\tau_{z}u)=\tau_{z}(Iu)
\]
for every $z$, where $\tau_{z}u(x):=u(x-z)$ is the translation operator.
In addition, without loss of generality we can assume that $I(0)=0$,
i.e. the action of $I$ on the constant function 0 is 0. Because by
translation invariance $I(0)$ is constant, and we can consider $I-I(0)$
instead of $I$.

Next let us recall some concepts from convex analysis. Let $K$ be
a compact convex subset of $\mathbb{R}^{n}$ whose interior contains
the origin.
\begin{defn}
The \textbf{gauge} function of $K$ is the function 
\begin{equation}
\gamma_{K}(x):=\inf\{\lambda>0:x\in\lambda K\}.\label{eq: gaug}
\end{equation}
And the \textbf{polar} of $K$ is the set 
\begin{equation}
K^{\circ}:=\{x:\langle x,y\rangle\leq1\,\textrm{ for all }y\in K\},\label{eq: K0}
\end{equation}
where $\langle\,,\rangle$ is the standard inner product on $\mathbb{R}^{n}$.
\end{defn}
The gauge function $\gamma_{K}$ is convex, subadditive, and positively
homogeneous; so it looks like a norm on $\mathbb{R}^{n}$, except
that $\gamma_{K}(-x)$ is not necessarily the same as $\gamma_{K}(x)$.
The polar set $K^{\circ}$ is also a compact convex set containing
the origin as an interior point. (For more details, and the proofs
of these facts, see \citep{MR3155183}.)

We assume that the exterior data $\varphi:\mathbb{R}^{n}\to\mathbb{R}$
is a bounded Lipschitz function which satisfies 
\begin{equation}
-\gamma_{K}(y-x)\le\varphi(x)-\varphi(y)\le\gamma_{K}(x-y),\label{eq: phi Lip}
\end{equation}
for all $x,y\in\R^{n}$. Then by Lemma 2.1 of \citep{MR1797872} this
property implies that $D\varphi\in K^{\circ}$ a.e. 

Let $U\subset\R^{n}$ be a bounded open set. Then for $x\in\overline{U}$
the obstacles are defined as
\begin{align}
 & \rho(x)=\rho_{K,\varphi}(x;U):=\underset{y\in\partial U}{\min}[\gamma_{K}(x-y)+\varphi(y)],\nonumber \\
 & \bar{\rho}(x)=\bar{\rho}_{K,\varphi}(x;U):=\underset{y\in\partial U}{\min}[\gamma_{K}(y-x)-\varphi(y)].\label{eq: rho}
\end{align}
It is well known (see {[}\citealp{MR667669}, Section 5.3{]}) that
$\rho$ is the unique viscosity solution of the Hamilton-Jacobi equation
\begin{equation}
\begin{cases}
\gamma_{K^{\circ}}(Dv)=1 & \textrm{in }U,\\
v=\varphi & \textrm{on }\partial U.
\end{cases}\label{eq: H-J eq}
\end{equation}
Now, note that $-K$ is also a compact convex set whose interior contains
the origin. We also have $\bar{\rho}_{K,\varphi}=\rho_{-K,-\varphi}$,
since $\gamma_{-K}(\cdot)=\gamma_{K}(-\,\cdot)$. Thus we have a similar
characterization for $\bar{\rho}$ too.
\begin{notation*}
To simplify the notation, we will use the following conventions 
\[
\gamma:=\gamma_{K},\qquad\gamma^{\circ}:=\gamma_{K^{\circ}},\qquad\bar{\gamma}:=\gamma_{-K}.
\]
Thus in particular we have $\bar{\gamma}(x)=\gamma(-x)$.
\end{notation*}
In \citep{SAFDARI202176} we have shown that $-\bar{\rho}\le\rho$,
and
\begin{equation}
-\gamma(x-y)\le\rho(y)-\rho(x)\le\gamma(y-x).\label{eq: rho Lip}
\end{equation}
The above inequality also holds if we replace $\rho,\gamma$ with
$\bar{\rho},\bar{\gamma}$. Thus in particular, $\rho,\bar{\rho}$
are Lipschitz continuous. We have also shown that $-\bar{\rho}=\varphi=\rho$
on $\partial U$. We extend $\rho,-\bar{\rho}$ to $\R^{n}$ by setting
them equal to $\varphi$ on $\R^{n}-U$. In other words we set 
\begin{equation}
\rho(x)=\begin{cases}
\underset{y\in\partial U}{\min}[\gamma(x-y)+\varphi(y)] & x\in U,\\
\varphi(x) & x\in\R^{n}-U,
\end{cases}\label{eq: rho ext}
\end{equation}
and similarly for $-\bar{\rho}$. Note that these extensions are continuous
functions, and using (\ref{eq: phi Lip}) we can easily show that
$\rho,-\bar{\rho}$ satisfy (\ref{eq: rho Lip}) for all $x,y\in\R^{n}$.

Motivated by the double obstacle problems arising from variational
problems, in \citep{SAFDARI2021358} we have studied (local) fully
nonlinear double obstacle problems of the form 
\[
\begin{cases}
\max\{\min\{F(D^{2}u),\;u+\bar{\rho}\},u-\rho\}=0 & \textrm{in }U,\\
u=\varphi & \textrm{on }\partial U,
\end{cases}
\]
and employed them to obtain the regularity of (local) elliptic equations
with gradient constraints. And motivated by this result, in \citep{Safd-Nonlocal-dbl-obst}
we proved the existence and $C^{1,\alpha}$ regularity of solutions
to nonlocal double obstacle problems of the form (\ref{eq: dbl obstcl}),
in which the obstacles are assumed to be semi-concave/convex functions.
As we will see, this regularity result will play a critical role in
obtaining the regularity of nonlocal equations with gradient constraints.

Now let us state our main results. We denote the Euclidean distance
to $\partial U$ by $d(\cdot)=d(\cdot,\partial U):=\min_{y\in\partial U}|\cdot-y|$.
First let us collect all the assumptions we made so far to facilitate
their referencing.
\begin{assumption}
\label{assu: all}We assume that 
\begin{enumerate}
\item[\upshape{(a)}] $I$ is a translation invariant operator which is uniformly elliptic
with respect to $\mathcal{L}_{0}$ with ellipticity constants $\lambda,\Lambda$,
and $0<s_{0}<s<1$. We also assume that $I(0)=0$.%
\begin{comment}
We can consider $-I-c$ in the equation and have $+\|I(0)\|_{L^{\infty}(U)}$
in the estimate ???
\end{comment}
\item[\upshape{(b)}] $U\subset\R^{n}$ is a bounded open set, and $\varphi:\mathbb{R}^{n}\to\mathbb{R}$
is bounded and satisfies (\ref{eq: phi Lip}). Also, $\varphi$ is
convex on a neighborhood of $\partial U$.
\item[\upshape{(c)}] $K\subset\R^{n}$ is a compact convex set whose interior contains
the origin, and the obstacles $\rho,-\bar{\rho}$ are defined by (\ref{eq: rho})
and extended as in (\ref{eq: rho ext}).
\item[\upshape{(d)}] The gradient constraint is $H:=\gamma^{\circ}-1$.
\end{enumerate}
\end{assumption}
\begin{rem*}
Note that we are not assuming any regularity about $\partial K$ or
$\partial K^{\circ}$. In particular, $\gamma^{\circ}$, which defines
the gradient constraint, need not be $C^{1}$ or strictly convex.
Also, the obstacles can be highly irregular. Furthermore, note that
any convex gradient constraint of the general form $\tilde{H}(Du)\le0$
for which the set $\{\tilde{H}(\cdot)\le0\}$ is bounded, and contains
a neighborhood of the origin (which is a natural requirement in these
problems), can be written in the form $\gamma^{\circ}-1$ with respect
to the convex set $K=\{\tilde{H}(\cdot)\le0\}^{\circ}$. (Note that
$\{\tilde{H}(\cdot)\le0\}=K^{\circ}$, because the double polar of
such convex sets are themselves, as shown in Section 1.6 of \citep{MR3155183}.)
\end{rem*}
\begin{thm}
\label{thm: Reg PDE grad}Suppose Assumption \ref{assu: all} holds.
Also suppose $\partial U$ is $C^{2}$, and $\varphi$ is $C^{2}$%
\begin{comment}
$C^{2,\alpha}$ ???
\end{comment}
{} with $\gamma^{\circ}(D\varphi)<1$. In addition, suppose there is
a bounded continuous function $-\bar{\rho}\le v\le\rho$ that satisfies
$-Iv\le0$ in the viscosity sense in $U$. Then the nonlocal elliptic
equation with gradient constraint (\ref{eq: PDE grad constr}) has
a viscosity solution $u$, and 
\[
u\in C_{\mathrm{loc}}^{1,\alpha}(U)\cap C^{0,1}(\overline{U})
\]
for some $\alpha>0$ depending only on $n,\lambda,\Lambda,s_{0}$.
And for an open subset $V\subset\subset U$ we have 
\begin{equation}
\|u\|_{C^{1,\alpha}(\overline{V})}\le C(\|u\|_{L^{\infty}(\mathbb{R}^{n})}+C_{V}),\label{eq: C1,a u}
\end{equation}
where $C$ depends only on $n,\lambda,\Lambda,s_{0}$, and $d(V,\partial U)$;
and $C_{V}$ depends only on these constants together with $K,\partial U,\varphi$.
\end{thm}
\begin{rem*}
Note that if the equation with gradient constraint (\ref{eq: PDE grad constr})
has a solution then we must have a subsolution ($-I\,\cdot\le0$)
in $U$ with $\gamma^{\circ}(D\,\cdot)\le1$, which easily implies
that $-\bar{\rho}\le\cdot\le\rho$. Thus the existence of $v$ is
a natural requirement. 
\end{rem*}
\begin{rem*}
Note that our estimate is uniform for $s>s_{0}$; so we can retrieve
the interior estimate for local equations with gradient constraints
as $s\to1$. In addition, note that unlike the results on local fully
nonlinear equations with gradient constraints in \citep{Hynd,SAFDARI2021358},
we do not have any convexity assumption about $I$.%
\begin{comment}
This is partly true since we work with viscosity solutions and do
not need Evans-Krylov-type estimates to have classical solutions ????
\end{comment}
\end{rem*}
We split the proof of Theorem \ref{thm: Reg PDE grad} into two parts.
In Theorem \ref{thm: Reg dbl obstcl} we show that there is $u\in C_{\mathrm{loc}}^{1,\alpha}(U)$
that satisfies the nonlocal double obstacle problem (\ref{eq: dbl obstcl}).
(Note that this result is stronger than the regularity result in \citep{Safd-Nonlocal-dbl-obst},
since here we do not have the assumption of semi-concavity/convexity
of the obstacles; an assumption which fails to hold when $\partial K$
is not smooth enough.) And in Theorem \ref{thm: dbl obstcl =00003D PDE grad}
we prove that $u$ must also satisfy the nonlocal equation with gradient
constraint (\ref{eq: PDE grad constr}).
\begin{thm}
\label{thm: dbl obstcl =00003D PDE grad}Suppose Assumption \ref{assu: all}
holds. Also suppose%
\begin{comment}
$\partial U$ is $C^{1}$ ??
\end{comment}
{} there is a bounded continuous function $-\bar{\rho}\le v\le\rho$
that satisfies $-Iv\le0$ in the viscosity sense in $U$. Let $u\in C^{1}(U)$
be a viscosity solution of the double obstacle problem 
\begin{equation}
\begin{cases}
\max\{\min\{-Iu,\;u+\bar{\rho}\},u-\rho\}=0 & \textrm{in }U,\\
u=\varphi & \textrm{in }\mathbb{R}^{n}-U.
\end{cases}\label{eq: dbl obstcl}
\end{equation}
Then $u$ is also a viscosity solution of the equation with gradient
constraint (\ref{eq: PDE grad constr}).
\end{thm}
\begin{thm}
\label{thm: Reg dbl obstcl}Suppose Assumption \ref{assu: all} holds.
Also suppose $\partial U$ is $C^{2}$, and $\varphi$ is $C^{2}$%
\begin{comment}
$C^{2,\alpha}$ ???
\end{comment}
{} with $\gamma^{\circ}(D\varphi)<1$. Then the double obstacle problem
(\ref{eq: dbl obstcl}) has a viscosity solution $u$, and 
\[
u\in C_{\mathrm{loc}}^{1,\alpha}(U)\cap C^{0,1}(\overline{U})
\]
for some $\alpha>0$ depending only on $n,\lambda,\Lambda,s_{0}$.
And for an open subset $V\subset\subset U$ the estimate (\ref{eq: C1,a u})
holds for $\|u\|_{C^{1,\alpha}(\overline{V})}$. 
\end{thm}
\textcolor{red}{}%
\begin{comment}
\textendash{} In Theorem \ref{thm: Reg dbl obstcl}, and hence in
Theorem \ref{thm: Reg PDE grad}, we can assume $\varphi$ is $C^{1,1}$.
In this case, in the proof of Theorem \ref{thm: Reg dbl obstcl} we
have to work with smooth approximations of $\varphi$ instead of $\varphi$.
The required modifications can be done easily; however, to keep the
presentation of the proof more clear we preferred to not follow this
route. ???

\textendash{} state this for the variational case in the other paper
- Does it hold for the 2nd theorem about convex domains too ??

\textendash{} say that $\gamma^{\circ}(D\varphi)<1$ ? Is this really
needed ? (apparently not)

\textendash{} say that $\partial U$ still needs to be $C^{2,\alpha}$,
since we need to construct $u_{k}$ by solving fully nonlinear equations
? ??
\end{comment}
\begin{comment}
\textendash{} WE CAN USE $-I-f$ in this theorem (NO, SINCE we use
Lemma \ref{prop: segment is plastic} and hence Lemma \ref{lem: g(Du)<1}??!!)
and $-I-c$ in theorem 1 ???? Although the equivalence does not hold
when there is $x$-dependence ?? 

\textendash{} Consider the case $I0\ne0$ and add $|I0|$ to the estimates
??
\end{comment}

The paper is organized as follows. In Section \ref{sec: Prelim} we
prove Theorem \ref{thm: dbl obstcl =00003D PDE grad}. Here we use
Lemma \ref{lem: g(Du)<1} in which we show that a solution to the
double obstacle problem (\ref{eq: dbl obstcl}) must also satisfy
the gradient constraint; and Lemma \ref{lem: u<v} in which we prove
that a solution to the double obstacle problem (\ref{eq: dbl obstcl})
must be larger than a solution to the elliptic equation with gradient
constraint (\ref{eq: PDE grad constr}).%
\begin{comment}
JUST say auxiliary function $v$ ??
\end{comment}
{} Then we review some well-known facts about the regularity of $K$,
and its relation to the regularity of $K^{\circ},\gamma,\gamma^{\circ}$.
After that we consider the function $\rho$ more carefully. We will
review the formulas for the derivatives of $\rho$ that we have obtained
in \citep{SAFDARI202176}, especially the novel explicit formula (\ref{eq: D2 rho (x)})
for $D^{2}\rho$. To the best of author's knowledge, formulas of this
kind have not appeared in the literature before, except for the simple
case where $\rho$ is the Euclidean distance to the boundary. (Although,
some special two dimensional cases also appeared in our earlier works
\citep{Safdari20151,Paper-4}.)

One of the main applications of the formula (\ref{eq: D2 rho (x)})
for $D^{2}\rho$ is in the relation (\ref{eq: ridge =00003D Q=00003D0})
for characterizing the set of singularities of $\rho$. Another important
application is in Lemma \ref{lem: D2 rho decreas}, which implies
that $D^{2}\rho$ attains its maximum on $\partial U$. This interesting
property is actually a consequence of a more general property of the
solutions to Hamilton-Jacobi equations (remember that $\rho$ is the
viscosity solution of the Hamilton-Jacobi equation (\ref{eq: H-J eq})).
This little-known monotonicity property is investigated in \citep{SAFDARI202176};
but we included a brief account at the end of Section \ref{sec: Prelim}
for reader's convenience.

In Section \ref{sec: Proof Thm 3} we prove the regularity result
for double obstacle problem (\ref{eq: dbl obstcl}), aka Theorem \ref{thm: Reg dbl obstcl},
when the exterior data $\varphi$ is zero. We separated this case
to reduce the technicalities so that the main ideas can be followed
more easily. The proof of the general case of nonzero exterior data
is postponed to Appendix \ref{sec: App A}. The idea of the proof
of Theorem \ref{thm: Reg dbl obstcl} is to approximate $K^{\circ}$
with smoother convex sets. Then we have to find uniform bounds for
the various norms of the approximations $u_{k}$ to $u$. In order
to do this, we construct appropriate barriers, which resemble the
obstacles $\rho,-\bar{\rho}$. Then, among other estimations, we will
use the fact that the second derivative of the barriers attain their
maximums on the boundary. A more detailed sketch of proof for Theorem
\ref{thm: Reg dbl obstcl} is given at the beginning of its proof
(before its Part I) on page \pageref{proof Thm 3}. %
\begin{comment}
\textcolor{red}{We should mention that our technique, even for local
regularity, is inherently global. BESIDES the globality of the equation
itself!!! Because we use the behavior of the obstacles at $\partial U$
in a crucial way. In particular we employ Lemma \ref{lem: D2 rho decreas},
which is a monotonicity property for $D^{2}\rho,D^{2}\bar{\rho}$.}
The methods of ??? are adapted to the setting of nonlocal equations
and their viscosity solutions. 
\end{comment}

\section{\label{sec: Prelim}Assumptions and Preliminaries}

Let us start by defining the notion of viscosity solutions of nonlocal
equations. We want to study the nonlocal double obstacle problems
and nonlocal equations with gradient constraints. So we state the
definition of viscosity solution in the more general case of a nonlocal
operator $\tilde{I}(x,u(x),Du(x),u(\cdot))$ whose value also depends
on the pointwise values of $x,u(x),Du(x)$ (see \citep{barles2008second}
for more details). For example, in the case of the double obstacle
problem and the equation with gradient constraint we respectively
have 
\begin{align*}
-\tilde{I}(x,r,u(\cdot)) & =\max\{\min\{-Iu(x),\;r-\psi^{-}(x)\},r-\psi^{+}(x)\},\\
-\tilde{I}(x,p,u(\cdot)) & =\max\{-Iu(x),\;H(p)\},
\end{align*}
where we replaced $u(x)$ with $r$ and $Du(x)$ with $p$ to clarify
the dependence of $\tilde{I}$ on its arguments.
\begin{defn}
An upper semi-continuous%
\begin{comment}
(USC)
\end{comment}
{} function $u:\R^{n}\to\R$ is a \textit{viscosity subsolution} of
$-\tilde{I}\le0$ in $U$ if whenever $\phi$ is a bounded $C^{2}$
function and $u-\phi$ has a maximum over $\R^{n}$ at $x_{0}\in U$
we have 
\[
-\tilde{I}(x_{0},u(x_{0}),D\phi(x),\phi(\cdot))\le0.
\]
And a lower semi-continuous%
\begin{comment}
(LSC)
\end{comment}
{} function $u$ is a \textit{viscosity supersolution} of $-\tilde{I}\ge0$
in $U$ if whenever $\phi$ is a bounded $C^{2}$ function and $u-\phi$
has a minimum over $\R^{n}$ at $x_{0}\in U$ we have 
\[
-\tilde{I}(x_{0},u(x_{0}),D\phi(x),\phi(\cdot))\ge0.
\]
A continuous function $u$ is a \textit{viscosity solution} of $-\tilde{I}=0$
in $U$ if it is a subsolution of $-\tilde{I}\le0$ and a supersolution
of $-\tilde{I}\ge0$ in $U$.
\end{defn}
\begin{lem}
\label{lem: g(Du)<1}Suppose Assumption \ref{assu: all} holds. Let
$u\in C^{1}(U)$ be a viscosity solution of the double obstacle problem
(\ref{eq: dbl obstcl}). Then over $U$ we have 
\[
\gamma^{\circ}(Du)\le1.
\]
\end{lem}
\begin{rem*}
An immediate consequence of this lemma is that $u\in C^{0,1}(\overline{U})$,
since $u$ and its derivative are bounded, and $\partial U$ is smooth
enough.
\end{rem*}
\begin{proof}
On the open set $V:=\{u<\rho\}$%
\begin{comment}
$\subset U$ ??
\end{comment}
{} we know that $-Iu\ge0$ in the viscosity sense. Also, on the open
set $V_{0}:=\{-\bar{\rho}<u\}$ we know that $-Iu\le0$ in the viscosity
sense. Hence due to the translation invariance of $I$, for any $h$
we have $-Iu(\cdot+h)\le0$ in the viscosity sense in $V_{0}-h$.
Therefore by Lemma 5.8 of \citep{caffarelli2009regularity} we have
\[
M^{+}(u(\cdot+h)-u(\cdot))\ge0
\]
in the viscosity sense in $V_{1}:=V\cap(V_{0}-h)$. Thus by Lemma
5.10 of \citep{caffarelli2009regularity} we have 
\begin{equation}
\sup_{V_{1}}(u(\cdot+h)-u(\cdot))\le\sup_{\R^{n}-V_{1}}(u(\cdot+h)-u(\cdot)).\label{eq: max prncpl Du}
\end{equation}
Let us estimate $u(\cdot+h)-u(\cdot)$ on $\R^{n}-V_{1}$. Let $x\in\R^{n}-V_{1}$.
Then either $x\notin V$, or $x+h\notin V_{0}$. Suppose $x\notin V$.
Then $u(x)=\rho(x)$ (note that outside of $U$ we also have $u=\varphi=\rho$).
So 
\[
u(x+h)-u(x)\le\rho(x+h)-\rho(x)\le\gamma(h).
\]
Next suppose $x+h\notin V_{0}$. Then $u(x+h)=-\bar{\rho}(x+h)$.
Hence we have 
\[
u(x+h)-u(x)\le-(\bar{\rho}(x+h)-\bar{\rho}(x))=\bar{\rho}(x)-\bar{\rho}(x+h)\le\bar{\gamma}(-h)=\gamma(h).
\]
Therefore by (\ref{eq: max prncpl Du}) for every $x\in V_{1}$, and
hence for every $x\in\R^{n}$, we have $u(x+h)-u(x)\le\gamma(h)$.
Hence for $t>0$ we get 
\[
\frac{u(x+th)-u(x)}{t}\le\frac{1}{t}\gamma(th)=\gamma(h).
\]
Thus at the points of differentiability of $u$ we have $\langle Du,h\rangle\le\gamma(h)$,
and consequently $\gamma^{\circ}(Du)\le1$ due to (\ref{eq: gen Cauchy-Schwartz 2}). 
\end{proof}
\begin{lem}
\label{lem: u<v}Suppose Assumption \ref{assu: all} holds. Also suppose%
\begin{comment}
$\partial U$ is $C^{1}$ ??
\end{comment}
{} there is a bounded continuous function $-\bar{\rho}\le v\le\rho$
that satisfies $-Iv\le0$ in the viscosity sense in $U$. Let $u$
be a viscosity solution of the double obstacle problem (\ref{eq: dbl obstcl}).
Then we have 
\[
v\le u.
\]
As a result we get 
\begin{equation}
\max\{-Iu,\;u-\rho\}=0\label{eq: obs}
\end{equation}
in the viscosity sense in $U$.
\end{lem}
\begin{rem*}
In fact, this lemma is still true if we replace $\rho,-\bar{\rho}$
by any other upper and lower obstacles which agree on $\R^{n}-U$.
We can also replace $I\,\cdot$ by $I\cdot-f$ for some continuous
function $f$. 
\end{rem*}
\begin{proof}
On the open set $V:=\{u<\rho\}\subset U$ we have $-Iu\ge0$ and $-Iv\le0$
in the sense of viscosity. Also note that on $\R^{n}-V$ we have $u=\rho\ge v$.
Hence by Theorem 5.2 of \citep{caffarelli2009regularity} (also see
Lemma 6.1 of \citep{mou2015uniqueness}) we have $u\ge v$ on $V$
too, as desired. 

Now let us prove (\ref{eq: obs}). Suppose $u-\phi$ has a global
maximum at $x_{0}\in U$. Then at $x_{0}$ we have 
\[
\max\{\min\{-I\phi(x_{0}),\;u+\bar{\rho}\},u-\rho\}\le0.
\]
We know that $-\bar{\rho}(x_{0})\le u(x_{0})\le\rho(x_{0})$. If $-\bar{\rho}(x_{0})<u(x_{0})$
then we must have $-I\phi(x_{0})\le0$. And if $-\bar{\rho}(x_{0})=u(x_{0})$
then $v(x_{0})=-\bar{\rho}(x_{0})=u(x_{0})$, since $-\bar{\rho}\le v\le u$.
Hence for every $x\in\R^{n}$ we have 
\[
v(x_{0})-\phi(x_{0})=u(x_{0})-\phi(x_{0})\ge u(x)-\phi(x)\ge v(x)-\phi(x).
\]
So $v-\phi$ has a global maximum at $x_{0}$, and therefore $-I\phi(x_{0})\le0$.
Thus in either case we have 
\[
\max\{-I\phi(x_{0}),\;u-\rho\}\le0.
\]

Next suppose $u-\phi$ has a global minimum at $x_{0}\in U$. Then
at $x_{0}$ we have 
\[
\max\{\min\{-I\phi(x_{0}),\;u+\bar{\rho}\},u-\rho\}\ge0.
\]
So if $u(x_{0})<\rho(x_{0})$ then we must have $-I\phi(x_{0})\ge0$,
which implies that 
\[
\max\{-I\phi(x_{0}),\;u-\rho\}\ge0.
\]
\end{proof}
\begin{proof}[\textbf{Proof of Theorem \ref{thm: dbl obstcl =00003D PDE grad}}]
\label{proof Thm 2} Suppose $u-\phi$ has a global minimum at $x_{0}\in U$.
Then Lemma \ref{lem: u<v} implies that at $x_{0}$ we have 
\[
\max\{-I\phi(x_{0}),\;u-\rho\}\ge0.
\]
We need to show that 
\[
\max\{-I\phi(x_{0}),\;\gamma^{\circ}(D\phi(x_{0}))-1\}\ge0.
\]
If $-I\phi(x_{0})\ge0$ then we have the desired. Otherwise we must
have $u(x_{0})=\rho(x_{0})$. Let $y_{0}\in\partial U$ be a $\rho$-closest
point to $x_{0}$, i.e. $\rho(x_{0})=\gamma(x_{0}-y_{0})+\varphi(y_{0})$.
Then by Lemma \ref{lem: segment to the closest pt}, $y_{0}$ is also
a $\rho$-closest point on $\partial U$ to $x_{0}+t(y_{0}-x_{0})$
for $t\in[0,1]$. So we have 
\begin{align*}
\rho(x_{0}+t(y_{0}-x_{0})) & =\gamma(x_{0}+t(y_{0}-x_{0})-y_{0})+\varphi(y_{0})\\
 & =(1-t)\gamma(x_{0}-y_{0})+\varphi(y_{0}).
\end{align*}
So $u(x_{0}+t(y_{0}-x_{0}))\le\rho(x_{0}+t(y_{0}-x_{0}))=(1-t)\gamma(x_{0}-y_{0})+\varphi(y_{0})$.
On the other hand, for small negative $t$ we have $\rho(x_{0}+t(y_{0}-x_{0}))\le(1-t)\gamma(x_{0}-y_{0})+\varphi(y_{0})$,
since $y_{0}$ may not be a $\rho$-closest point on $\partial U$
to $x_{0}+t(y_{0}-x_{0})$. Hence for $t$ near $0$ we have 
\[
u(x_{0}+t(y_{0}-x_{0}))\le(1-t)\gamma(x_{0}-y_{0})+\varphi(y_{0}).
\]
And the equality holds at $t=0$. Thus by differentiating we get $\langle Du(x_{0}),x_{0}-y_{0}\rangle=\gamma(x_{0}-y_{0})$.
Therefore by (\ref{eq: gen Cauchy-Schwartz 2}) we get $\gamma^{\circ}(D\phi(x_{0}))\ge1$,
as desired.

Next suppose $u-\phi$ has a global maximum at $x_{0}\in U$. We need
to show that 
\[
\max\{-I\phi(x_{0}),\;\gamma^{\circ}(D\phi(x_{0}))-1\}\le0.
\]
By Lemma \ref{lem: u<v} we know that at $x_{0}$ we have 
\[
\max\{-I\phi(x_{0}),\;u-\rho\}\le0.
\]
So $-I\phi(x_{0})\le0$; and we only need to show that $\gamma^{\circ}(D\phi(x_{0}))\le1$.
However we know that 
\[
D\phi(x_{0})=Du(x_{0}),
\]
since $u$ is $C^{1}$ in $U$. Hence we get the desired by Lemma
\ref{lem: g(Du)<1}.
\end{proof}

\subsection{\label{subsec: Reg gaug}Regularity of the gauge function}

In this subsection we review some of the properties of $\gamma,\gamma^{\circ}$
and $K,K^{\circ}$ briefly. For detailed explanations and proofs see
\citep{MR3155183}. Recall that the gauge function $\gamma$ satisfies
\begin{align*}
 & \gamma(rx)=r\gamma(x),\\
 & \gamma(x+y)\le\gamma(x)+\gamma(y),
\end{align*}
for all $x,y\in\mathbb{R}^{n}$ and $r\ge0$. Also, note that as $B_{1/C}(0)\subseteq K\subseteq B_{1/c}(0)$
for some $C\ge c>0$, we have 
\begin{equation}
c|x|\le\gamma(x)\le C|x|\label{eq: c<g<C}
\end{equation}
for all $x\in\mathbb{R}^{n}$. In addition, since $K$ is closed we
have $K=\{\gamma\le1\}$, and since $K$ has nonempty interior we
have $\partial K=\{\gamma=1\}$. 

It is well known that for all $x,y\in\mathbb{R}^{n}$ we have 
\begin{equation}
\langle x,y\rangle\leq\gamma(x)\gamma^{\circ}(y).\label{eq: gen Cauchy-Schwartz}
\end{equation}
In fact, more is true and we have 
\begin{equation}
\gamma^{\circ}(y)=\underset{x\ne0}{\max}\frac{\langle x,y\rangle}{\gamma(x)}.\label{eq: gen Cauchy-Schwartz 2}
\end{equation}
For a proof of this see page 54 of \citep{MR3155183}.

It is easy to see that the the strict convexity of $K$ (which means
that $\partial K$ does not contain any line segment) is equivalent
to the strict convexity of $\gamma$. By homogeneity of $\gamma$
the latter is equivalent to 
\[
\gamma(x+y)<\gamma(x)+\gamma(y)
\]
when $x\ne cy$ and $y\ne cx$ for any $c\ge0$.

Suppose that $\partial K$ is $C^{k,\alpha}$ $(k\ge1\,,\,0\le\alpha\le1)$.
Then $\gamma$ is $C^{k,\alpha}$ on $\mathbb{R}^{n}-\{0\}$ (see
for example \citep{SAFDARI202176}). Conversely, note that as $\partial K=\{\gamma=1\}$
and $D\gamma\ne0$ by (\ref{eq: g0 (Dg)=00003D1}), $\partial K$
is as smooth as $\gamma$. Suppose in addition that $K$ is strictly
convex. Then $\gamma$ is strictly convex too. By Remark 1.7.14 and
Theorem 2.2.4 of \citep{MR3155183}, $K^{\circ}$ is also strictly
convex and its boundary is $C^{1}$. Therefore $\gamma^{\circ}$ is
strictly convex, and it is $C^{1}$ on $\mathbb{R}^{n}-\{0\}$. Furthermore,
by Corollary 1.7.3 of \citep{MR3155183}, for $x\ne0$ we have 
\begin{eqnarray}
D\gamma(x)\in\partial K^{\circ}, &  & D\gamma^{\circ}(x)\in\partial K,\label{eq: g0 (Dg)=00003D1}
\end{eqnarray}
or equivalently 
\[
\gamma^{\circ}(D\gamma)=1,\qquad\gamma(D\gamma^{\circ})=1.
\]
In particular $D\gamma,D\gamma^{\circ}$ are nonzero on $\mathbb{R}^{n}-\{0\}$.

Now assume that $k\ge2$ and the principal curvatures of $\partial K$
are positive everywhere. Then $K$ is strictly convex. We can also
show that $\gamma^{\circ}$ is $C^{k,\alpha}$ on $\mathbb{R}^{n}-\{0\}$.
To see this let $n_{K}:\partial K\to\mathbb{S}^{n-1}$ be the Gauss
map, i.e. let $n_{K}(y)$ be the outward unit normal to $\partial K$
at $y$. Then $n_{K}$ is $C^{k-1,\alpha}$ and its derivative is
an isomorphism at the points with positive principal curvatures, i.e.
everywhere. Hence $n_{K}$ is locally invertible with a $C^{k-1,\alpha}$
inverse $n_{K}^{-1}$, around any point of $\mathbb{S}^{n-1}$. Now
note that as it is well known, $\gamma^{\circ}$ equals the support
function of $K$, i.e. 
\[
\gamma^{\circ}(x)=\sup\{\langle x,y\rangle:y\in K\}.
\]
Thus as shown in page 115 of \citep{MR3155183}, for $x\ne0$ we have
\[
D\gamma^{\circ}(x)=n_{K}^{-1}(\frac{x}{|x|}).
\]
Which gives the desired result. As a consequence, $\partial K^{\circ}$
is $C^{k,\alpha}$ too. Furthermore, as shown on page 120 of \citep{MR3155183},
the principal curvatures of $\partial K^{\circ}$ are also positive
everywhere.

Let us recall a few more properties of $\gamma,\gamma^{\circ}$. Since
they are positively 1-homogeneous, $D\gamma,D\gamma^{\circ}$ are
positively 0-homogeneous, and $D^{2}\gamma,D^{2}\gamma^{\circ}$ are
positively $(-1)$-homogeneous, i.e. 
\begin{eqnarray}
\gamma(tx)=t\gamma(x), & D\gamma(tx)=D\gamma(x), & D^{2}\gamma(tx)=\frac{1}{t}D^{2}\gamma(x),\nonumber \\
\gamma^{\circ}(tx)=t\gamma^{\circ}(x), & D\gamma^{\circ}(tx)=D\gamma^{\circ}(x), & D^{2}\gamma^{\circ}(tx)=\frac{1}{t}D^{2}\gamma^{\circ}(x),\label{eq: homog}
\end{eqnarray}
for $x\ne0$ and $t>0$. As a result, using Euler's theorem on homogeneous
functions we get 
\begin{eqnarray}
\langle D\gamma(x),x\rangle=\gamma(x), &  & D^{2}\gamma(x)\,x=0,\nonumber \\
\langle D\gamma^{\circ}(x),x\rangle=\gamma^{\circ}(x), &  & D^{2}\gamma^{\circ}(x)\,x=0,\label{eq: Euler formula}
\end{eqnarray}
for $x\ne0$. Here $D^{2}\gamma(x)\,x$ is the action of the matrix
$D^{2}\gamma(x)$ on the vector $x$. 

Finally let us mention that by Corollary 2.5.2 of \citep{MR3155183},
when $x\ne0$, the eigenvalues of $D^{2}\gamma(x)$ are $0$ with
the corresponding eigenvector $x$, and $\frac{1}{|x|}$ times the
principal radii of curvature of $\partial K^{\circ}$ at the unique
point that has $x$ as an outward normal vector. Remember that the
principal radii of curvature are the reciprocals of the principal
curvatures. Thus by our assumption, the eigenvalues of $D^{2}\gamma(x)$
are all positive except for one $0$. We have a similar characterization
of the eigenvalues of $D^{2}\gamma^{\circ}(x)$.

\subsection{\label{subsec: Reg Opstacles}Regularity of the obstacles}

Next let us consider the obstacles $\rho,-\bar{\rho}$, and review
some of their properties. All the results of this subsection are proved
in \citep{SAFDARI202176}.
\begin{defn}
When $\rho(x)=\gamma(x-y)+\varphi(y)$ for some $y\in\partial U$,
we call $y$ a \textbf{$\boldsymbol{\rho}$-closest} point to $x$
on $\partial U$. Similarly, when $\bar{\rho}(x)=\gamma(y-x)-\varphi(y)$
for some $y\in\partial U$, we call $y$ a \textbf{$\boldsymbol{\bar{\rho}}$-closest}
point to $x$ on $\partial U$.
\end{defn}
Let us also introduce some more notation. For two points $x,y\in\R^{n}$,
$[x,y],\,]x,y[,\,[x,y[,\,]x,y]$ will denote the closed, open, and
half-open line segments with endpoints $x,y$, respectively. 
\begin{lem}
\label{lem: segment to the closest pt}Suppose $y$ is one of the
$\rho$-closest points on $\partial U$ to $x\in U$. Then 
\begin{enumerate}
\item[\upshape{(a)}] $y$ is a $\rho$-closest point on $\partial U$ to every point of
$]x,y[$. Therefore $\rho$ varies linearly along the line segment
$[x,y]$.
\item[\upshape{(b)}] If in addition, for all $x\ne y\in\R^{n}$ we have 
\begin{equation}
-\gamma(y-x)<\varphi(x)-\varphi(y)<\gamma(x-y),\label{eq: phi strct Lip}
\end{equation}
then we also have $]x,y[\subset U$.
\item[\upshape{(c)}] If in addition $\gamma$ is strictly convex, and the strict Lipschitz
property (\ref{eq: phi strct Lip}) for $\varphi$ holds, then $y$
is the unique $\rho$-closest point on $\partial U$ to the points
of $]x,y[$.
\end{enumerate}
\end{lem}
Next, we generalize the notion of ridge introduced by \citet{MR0184503},
and \citet{MR534111}. Intuitively, the $\rho$-ridge is the set of
singularities of $\rho$.
\begin{defn}
\label{def: ridge}The \textbf{$\boldsymbol{\rho}$-ridge} of $U$
is the set of all points $x\in U$ where $\rho(x)$ is not $C^{1,1}$
in any neighborhood of $x$. We denote it by 
\[
R_{\rho}.
\]
We have shown that when $\gamma$ is strictly convex and the strict
Lipschitz property (\ref{eq: phi strct Lip}) for $\varphi$ holds,
the points with more than one $\rho$-closest point on $\partial U$
belong to $\rho$-ridge, since $\rho$ is not differentiable at them.
This subset of the $\rho$-ridge is denoted by 
\[
R_{\rho,0}.
\]
Similarly we define $R_{\bar{\rho}},R_{\bar{\rho},0}$.%
\begin{comment}
\textendash{} mention other names of ridge, like cut locus ??

\textendash{} $R_{-\bar{\rho}}$ ??

\textendash{} $R_{\rho}$ is closed in $U$ ?
\end{comment}
\end{defn}
We know that $\rho,\bar{\rho}$ are Lipschitz functions. We want to
characterize the set over which they are more regular. In order to
do that, we need to impose some additional restrictions on $K,U$
and $\varphi$.
\begin{assumption}
\label{assu: K,U}Suppose that $k\ge2$ is an integer, and $0\le\alpha\le1$.
We assume that 
\begin{enumerate}
\item[\upshape{(a)}] $K\subset\R^{n}$ is a compact convex set whose interior contains
the origin. In addition, $\partial K$ is $C^{k,\alpha}$, and its
principal curvatures are positive at every point.
\item[\upshape{(b)}] $U\subset\R^{n}$ is a bounded open set, and $\partial U$ is $C^{k,\alpha}$.
\item[\upshape{(c)}] $\varphi:\R^{n}\to\R$ is a $C^{k,\alpha}$ function, such that $\gamma^{\circ}(D\varphi)<1$.
\end{enumerate}
\end{assumption}
\begin{rem*}
As shown in Subsection \ref{subsec: Reg gaug}, the above assumption
implies that $K,\gamma$ are strictly convex. In addition, $K^{\circ},\gamma^{\circ}$
are strictly convex, and $\partial K^{\circ},\gamma^{\circ}$ are
also $C^{k,\alpha}$. Furthermore, the principal curvatures of $\partial K^{\circ}$
are also positive at every point. Similar conclusions obviously hold
for $-K,-\varphi$ and $(-K)^{\circ}=-K^{\circ}$ too. Hence in the
sequel, whenever we state a property for $\rho$, it holds for $\bar{\rho}$
too.\textcolor{red}{}
\end{rem*}
Let $\nu$ be the inward unit normal to $\partial U$. Then for every
$y\in\partial U$ there is a unique scalar $\lambda(y)>0$ such that
\begin{equation}
\gamma^{\circ}\big(D\varphi(y)+\lambda(y)\nu(y)\big)=1.\label{eq: lambda}
\end{equation}
In addition, $\lambda$ is a $C^{k-1,\alpha}$ function of $y$. We
set 
\begin{equation}
\mu(y):=D\varphi(y)+\lambda(y)\nu(y).\label{eq: mu}
\end{equation}
We also set 
\begin{equation}
X:=\frac{1}{\langle D\gamma^{\circ}(\mu),\nu\rangle}D\gamma^{\circ}(\mu)\otimes\nu,\label{eq: X}
\end{equation}
where $a\otimes b$ is the rank 1 matrix whose action on a vector
$z$ is $\langle z,b\rangle a$. We also know that 
\begin{equation}
\langle D\gamma^{\circ}(\mu),\nu\rangle>0.\label{eq: Dg(mu) . nu >0}
\end{equation}

Let $x\in U$, and suppose $y$ is one of the $\rho$-closest points
to $x$ on $\partial U$. Then we have 
\begin{equation}
\frac{x-y}{\gamma(x-y)}=D\gamma^{\circ}(\mu(y))\label{eq: K-normal}
\end{equation}
(see Figure \ref{fig: 1}), or equivalently 
\begin{equation}
x=y+\big(\rho(x)-\varphi(y)\big)\,D\gamma^{\circ}(\mu(y)).\label{eq: parametrize by rho}
\end{equation}
Also, $\rho$ is differentiable at $x$ if and only if $x\in U-R_{\rho,0}$.
And in that case we have 
\begin{equation}
D\rho(x)=\mu(y),\label{eq: D rho (x)}
\end{equation}
where $y$ is the unique $\rho$-closest point to $x$ on $\partial U$.

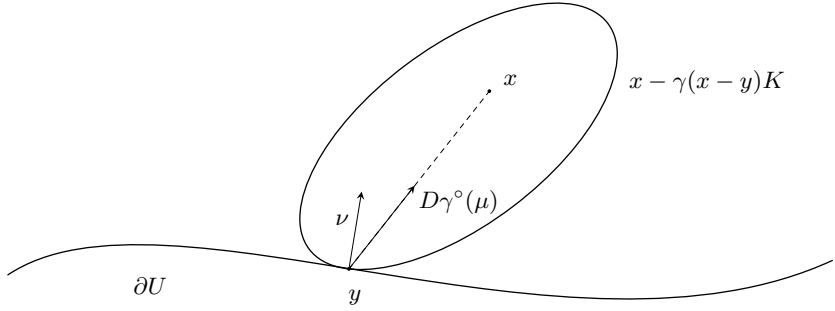
\begin{figure}

\begin{tikzpicture}[line cap=round,line join=round,>=triangle 45,x=1.0cm,y=1.0cm] 

\clip(-5,-1) rectangle (8.,4.5); 

\draw [rotate around={-30:(0.74,0.509615431964962)}, rotate around={-112.9124476060207:(1.1199857406279552,2.7658228013209634)},line width=0.5pt] (1.1199857406279552,2.7658228013209634) ellipse (2.489354745523695cm and 1.1806387377334575cm); 

\draw [rotate around={-21:(0.74,0.509615431964962)}, line width=0.3pt,dash pattern=on 2pt off 2pt] (0.74,0.509615431964962)-- (1.6350559584448812,3.3863653478797637); 

\draw [rotate around={-21:(0.74,0.509615431964962)}, ->,line width=0.3pt, >=stealth] (0.74,0.509615431964962) -- (1.1591475422042112,1.8567743949496367); 

\draw [rotate around={-28:(0.74,0.509615431964962)}, ->,line width=0.3pt, >=stealth] (0.74,0.509615431964962) -- (0.40621183416745316,1.4974466417936942); 

\draw [rotate around={-18:(0.74,0.509615431964962)}, line width=0.5pt]  (-3.5479974267239784,-0.9673560492307696).. controls (-2,1) and (4,0) ..  (6.8174727583718076,2.6029594918089893);

\begin{scriptsize} 

\draw [fill=black] (0.74,0.509615431964962) circle (0.5pt); 

\draw[color=black] (0.82,0.13) node {$y$};

\draw[rotate around={-21:(0.74,0.509615431964962)}, color=black] (1.84,3.63) node {$x$}; 

\draw [rotate around={-21:(0.74,0.509615431964962)}, fill=black] (1.6350559584448812,3.3863653478797637) circle (0.5pt);

\draw[color=black] (2.2,1.4) node {$D\gamma ^\circ (\mu)$}; 

\draw[rotate around={-30:(0.74,0.509615431964962)}, color=black] (0.34,1.03) node {$\nu$};

\draw[rotate around={-30:(0.74,0.509615431964962)}, color=black] (-1.44,-1) node {$\partial U$}; 

\draw[color=black] (5.5,3) node {$x - \gamma (x-y) K$}; 

\end{scriptsize} 
\end{tikzpicture} 

\caption{$y$ is a $\rho$-closest point to $x$.}
\label{fig: 1}
\end{figure}

In addition, for every $y\in\partial U$ there is an open ball $B_{r}(y)$
such that $\rho$ is $C^{k,\alpha}$ on $\overline{U}\cap B_{r}(y)$.
Furthermore, $y$ is the $\rho$-closest point to some points in $U$,
and we have 
\begin{equation}
D\rho(y)=\mu(y).\label{eq: D rho (y)}
\end{equation}
We also have 
\begin{equation}
D^{2}\rho(y)=(I-X^{T})\big(D^{2}\varphi(y)+\lambda(y)D^{2}d(y)\big)(I-X),\label{eq: D2 rho (y)}
\end{equation}
where $I$ is the identity matrix, $d$ is the Euclidean distance
to $\partial U$, and $X$ is given by (\ref{eq: X}).
\begin{rem*}
As a consequence, $R_{\rho}$ has a positive distance from $\partial U$.
\end{rem*}
Let $x\in U-R_{\rho,0}$, and let $y$ be the unique $\rho$-closest
point to $x$ on $\partial U$. Let 
\begin{align}
 & W=W(y):=-D^{2}\gamma^{\circ}(\mu(y))D^{2}\rho(y),\nonumber \\
 & Q=Q(x):=I-\big(\rho(x)-\varphi(y)\big)W,\label{eq: W,Q}
\end{align}
where $I$ is the identity matrix. If $\det Q\ne0$ then $\rho$ is
$C^{k,\alpha}$ on a neighborhood of $x$. In addition we have 
\begin{equation}
D^{2}\rho(x)=D^{2}\rho(y)Q(x)^{-1}.\label{eq: D2 rho (x)}
\end{equation}
We also have 
\begin{equation}
x\in R_{\rho}\textrm{ if and only if }\det Q(x)=0.\label{eq: ridge =00003D Q=00003D0}
\end{equation}

\begin{rem*}
When $\varphi=0$, the function $\rho$ is the distance to $\partial U$
with respect to the Minkowski distance defined by $\gamma$. So this
case has a geometric interpretation. An interesting fact is that in
this case the eigenvalues of $W$ coincide with the notion of curvature
of $\partial U$ with respect to some Finsler structure. For the details
see \citep{MR2336304}.
\end{rem*}
\begin{lem}
\label{lem: D2 rho decreas}Suppose the Assumption \ref{assu: K,U}
holds. Let $x\in U-R_{\rho}$, and let $y$ be the unique $\rho$-closest
point to $x$ on $\partial U$. Then we have 
\[
D_{\xi\xi}^{2}\rho(x)\le D_{\xi\xi}^{2}\rho(y)
\]
for every $\xi\in\R^{n}$.
\end{lem}
As we mentioned in the introduction, the above monotonicity property
is true because $\rho$ satisfies the Hamilton-Jacobi equation (\ref{eq: H-J eq}),
and  the segment $]x,y[$ is the characteristic curve associated
to it. Let us review the general case of the monotonicity property
below. (Although note that the assumptions in Lemma \ref{lem: D2 rho decreas}
are weaker than what we assume in the following calculations. For
a proof of Lemma \ref{lem: D2 rho decreas} see the proof of Lemma
4 in \citep{SAFDARI202176}.)%
\begin{comment}
UPDATE the lemma no. after publication ??
\end{comment}

\textbf{Monotonicity of the second derivative of the solutions to
Hamilton-Jacobi equations:} Suppose $v$ satisfies the equation $\tilde{H}(Dv)=0$,
where $\tilde{H}$ is a convex function. Let $x(s)$ be a characteristic
curve of the equation. Then we have $\dot{x}=D\tilde{H}$. Let us
assume that $v$ is $C^{3}$ on a neighborhood of the image of $x(s)$.
For some vector $\xi$ let 
\[
q(s):=D_{\xi\xi}^{2}v(x(s))=\xi_{i}\xi_{j}D_{ij}^{2}v.
\]
Here we have used the convention of summing over repeated indices.
Then we have 
\[
\dot{q}=\xi_{i}\xi_{j}D_{ijk}^{3}v\,\dot{x}^{k}=\xi_{i}\xi_{j}D_{ijk}^{3}vD_{k}\tilde{H}.
\]
On the other hand, if we differentiate the equation we get $D_{k}\tilde{H}D_{ik}^{2}v=0$.
And if we differentiate one more time we get 
\[
D_{kl}^{2}\tilde{H}D_{jl}^{2}vD_{ik}^{2}v+D_{k}\tilde{H}D_{ijk}^{3}v=0.
\]
Now if we multiply the above expression by $\xi_{i}\xi_{j}$, and
sum over $i,j$, we obtain the following Riccati type equation 
\begin{equation}
\dot{q}=-\xi^{T}D^{2}v\,D^{2}\tilde{H}\,D^{2}v\xi.\label{eq: ODE D2 rho}
\end{equation}
So $\dot{q}\le0$, since $\tilde{H}$ is convex. Thus we have 
\[
D_{\xi\xi}^{2}v(x(s))=q(s)\le q(0)=D_{\xi\xi}^{2}v(x(0)),
\]
as desired. This result also holds in the more general case of $\tilde{H}(x,v,Dv)=0$,
when $\tilde{H}$ is a convex function in all of its arguments (see
\citep{SAFDARI202176}).

\section{\label{sec: Proof Thm 3}Proof of Theorem \ref{thm: Reg dbl obstcl}}

In this section we prove Theorem \ref{thm: Reg dbl obstcl}, i.e.
we will prove that the double obstacle problem (\ref{eq: dbl obstcl})
has a viscosity solution $u\in C_{\mathrm{loc}}^{1,\alpha}(U)$, without
assuming any regularity about $K$. Here we only consider the case
of zero exterior data. The case of general nonzero exterior data $\varphi$
is considered in Appendix \ref{sec: App A}.
\begin{proof}[\textbf{Proof of Theorem \ref{thm: Reg dbl obstcl} for zero exterior
data}]
\label{proof Thm 3} As it is well known, a compact convex set with
nonempty interior can be approximated, in the Hausdorff metric, by
a shrinking sequence of compact convex sets with nonempty interior
which have smooth boundaries with positive curvature (see for example
\citep{schmuckenschlaeger1993simple}). We apply this result to $K^{\circ}$.
Thus there is a sequence $K_{k}^{\circ}$ of compact convex sets,
that have smooth boundaries with positive curvature, and 
\begin{eqnarray*}
K_{k+1}^{\circ}\subset\mathrm{int}(K_{k}^{\circ}), & \qquad & K^{\circ}={\textstyle \bigcap}K_{k}^{\circ}.
\end{eqnarray*}
Notice that we can take the approximations of $K^{\circ}$ to be
the polar of other convex sets, because the double polar of a compact
convex set with $0$ in its interior is itself. Also note that $K_{k}$'s
are strictly convex compact sets with $0$ in their interior, which
have smooth boundaries with positive curvature. Furthermore we have
$K=(K^{\circ})^{\circ}\supset K_{k+1}\supset K_{k}$. For the proof
of these facts see {[}\citealp{MR3155183}, Sections 1.6, 1.7 and
2.5{]}.

To simplify the notation we use $\gamma_{k},\gamma_{k}^{\circ},\rho_{k},\bar{\rho}_{k}$
instead of $\gamma_{K_{k}},\gamma_{K_{k}^{\circ}},\rho_{K_{k},0},\bar{\rho}_{K_{k},0}$,
respectively. Note that $K_{k},U,\varphi=0$ satisfy the Assumption
\ref{assu: K,U}. Hence as we have shown in \citep{SAFDARI202176},
$\rho_{k},\bar{\rho}_{k}$ satisfy the assumptions of Theorem 1 of
\citep{Safd-Nonlocal-dbl-obst}. Thus there are viscosity solutions
$u_{k}\in C_{\mathrm{loc}}^{1,\alpha}(U)$ of the double obstacle
problem 
\begin{equation}
\begin{cases}
\max\{\min\{-Iu_{k},\;u_{k}+\bar{\rho}_{k}\},u_{k}-\rho_{k}\}=0 & \textrm{in }U,\\
u_{k}=0 & \textrm{in }\mathbb{R}^{n}-U.
\end{cases}\label{eq: dbl obstcl k}
\end{equation}
And $\alpha>0$ depends only on $n,\lambda,\Lambda,s_{0}$. 

In addition, we know that 
\begin{equation}
-\bar{\rho}_{1}\le-\bar{\rho}_{k}\le u_{k}\le\rho_{k}\le\rho_{1}.\label{eq: u_k bdd}
\end{equation}
Note that $\rho_{k}\le\rho_{1}$ and $\bar{\rho}_{k}\le\bar{\rho}_{1}$,
since $\gamma_{k}\le\gamma_{1}$ due to $K_{k}\supset K_{1}$. Let
us also show that $\rho_{k}\to\rho$ and $\bar{\rho}_{k}\to\bar{\rho}$
uniformly on $U$. It is easy to see that for nonzero $z$ we have
$\frac{z}{\gamma(z)}\in\partial K$. Each ray emanating from the origin
intersects $\partial K_{k},\partial K$ at a pair of points. Let $\delta_{k}$
be the maximum distance between these pairs of points. Then we have
\[
\big|\frac{z}{\gamma(z)}-\frac{z}{\gamma_{k}(z)}\big|\le\delta_{k}\implies|\gamma_{k}(z)-\gamma(z)|\le\delta_{k}\frac{\gamma(z)\gamma_{k}(z)}{|z|}\le C\delta_{k}|z|,
\]
where the last estimate is obtained by (\ref{eq: c<g<C}). Now let
$x\in U$, and let $y\in\partial U$ be a $\rho$-closest point to
$x$. Then we have 
\begin{equation}
|\rho_{k}(x)-\rho(x)|\le|\gamma_{k}(x-y)-\gamma(x-y)|\le C\delta_{k}|x-y|.\label{eq: rho_k - rho}
\end{equation}
Since $\delta_{k}\to0$ and $U$ is bounded we get the desired.

Furthermore, by Lemma \ref{lem: g(Du)<1} we have $\gamma_{k}^{\circ}(Du_{k})\le1$.
Hence we have 
\[
Du_{k}\in K_{k}^{\circ}\subset K_{1}^{\circ}\quad\textrm{ a.e.}
\]
Therefore $u_{k}$ is a bounded sequence in $W^{1,\infty}(U)=C^{0,1}(\overline{U})$.
Hence by the Arzela-Ascoli Theorem a subsequence of $u_{k}$, which
we still denote by $u_{k}$, uniformly converges to a continuous function
$u\in C^{0}(\overline{U})$. Note that $u|_{\partial U}=0$, because
$u_{k}|_{\partial U}=0$ for every $k$. We extend $u$ to all of
$\R^{n}$ by setting it equal to $0$ in $\R^{n}-U$. Note that $u$
is a continuous function. 

We divide the rest of this proof into three parts. In Part I we derive
the uniform bound (\ref{eq: I(u_e) bdd}), i.e. we show that $Iu_{k}$
is bounded independently of $k$ by using suitable barriers. This
is possible mainly for two reasons. First we will use the fact that
similarly to $D^{2}\rho_{k}$, the second derivative of the barrier
attains its maximum on the boundary; so we get a one-way bound for
$Iu_{k}$. For the other bound, we use the fact that $u_{k}$ is a
subsolution or a supersolution of $-I=0$ in the regions in which
$u_{k}$ touches one of the obstacles. In Part II we prove the main
properties of the barriers which we employed in Part I. Then in Part
III we show that a subsequence of $u_{k}$ converges to $u$ in $C_{\mathrm{loc}}^{1,\alpha}$,
and $u$ is a viscosity solution of the double obstacle problem (\ref{eq: dbl obstcl}).
Here we use the bound (\ref{eq: I(u_e) bdd}), obtained in Part I,
to show that the $C_{\mathrm{loc}}^{1,\alpha}$ norm of $u_{k}$ is
uniformly bounded, so that we can extract a convergent subsequence
of $u_{k}$.

\medskip{}

PART I:\medskip{}

Let us show that for every bounded open set $V\subset\subset U$ and
every $k$ we have 
\begin{equation}
-C_{V}\le Iu_{k}\le C_{V}\label{eq: I(u_e) bdd}
\end{equation}
in the viscosity sense in $V$, for some constant $C_{V}$ independent
of $k$. In addition, we can choose $C_{V}$ uniformly for $s>s_{0}$.

Suppose $\phi$ is a bounded $C^{2}$ function and $u_{k}-\phi$ has
a minimum over $\R^{n}$ at $x_{0}\in V$. We must show that 
\[
-I\phi(x_{0})\ge-C_{V}.
\]
We can assume that $u_{k}(x_{0})-\phi(x_{0})=0$ without loss of generality,
since we can consider $\phi+c$ instead of $\phi$ without changing
$I$ (because $M^{\pm}(c)=0$). So we can assume that $u_{k}-\phi\ge0$,
or $u_{k}\ge\phi$. We also know that at $x_{0}$ we have 
\[
\max\{\min\{-I\phi(x_{0}),u_{k}+\bar{\rho}_{k}\},u_{k}-\rho_{k}\}\ge0,
\]
since $u_{k}$ is a viscosity solution of (\ref{eq: dbl obstcl k}).
In addition remember that $-\bar{\rho}_{k}\leq u_{k}\leq\rho_{k}$.
Now if $u_{k}(x_{0})<\rho_{k}(x_{0})$ then we must have $-I\phi(x_{0})\ge0$.
And if $u_{k}(x_{0})=\rho_{k}(x_{0})$ then $\phi$ is also touching
$\rho_{k}$ from below at $x_{0}$, since $\phi\le u_{k}\le\rho_{k}$.
Now let $y_{0}\in\partial U$ be a $\rho_{k}$-closest point to
$x_{0}$. (Note that $y_{0}$ can depend on $k$.) Let $B\subset\R^{n}-\overline{U}$
be an open ball such that $\partial B\cap\partial U=\{y_{0}\}$. Note
that the radius of $B$ can be chosen to be independent of $y_{0}$,
since $\partial U$ is $C^{2}$. Now on the domain $\R^{n}-\overline{B}$
consider the function 
\begin{equation}
\rho_{k,B^{c}}:=\rho_{K_{k},0}(\cdot\,;\R^{n}-\overline{B})=\underset{z\in\partial B}{\min}\gamma_{k}(\cdot-z).\label{eq: barrier}
\end{equation}
The idea is to use $\rho_{k,B^{c}}$ as a barrier to obtain the desired
bound for $-I\phi$.

Let $z\in\partial B$, and let $y$ be a point on $\partial U\cap[z,x_{0}[$.
Then we have 
\[
\gamma_{k}(x_{0}-y_{0})\le\gamma_{k}(x_{0}-y)=\gamma_{k}(x_{0}-z)-\gamma_{k}(y-z)\le\gamma_{k}(x_{0}-z).
\]
Hence $y_{0}$ is also a $\rho_{k,B^{c}}$-closest point to $x_{0}$
on $\partial B$.%
\begin{comment}
$y$ can also be equal to $y_{0}$
\end{comment}
{} Consequently we have 
\[
\rho_{k,B^{c}}(x_{0})=\gamma_{k}(x_{0}-y_{0})=\rho_{k}(x_{0}).
\]
Now consider $x\in U\subset\R^{n}-\overline{B}$, and let $z\in\partial B$
be a $\rho_{k,B^{c}}$-closest point to $x$. Let $y$ be a point
on $\partial U\cap[z,x[$. Then similarly to the above we can show
that 
\[
\rho_{k}(x)\le\gamma_{k}(x-y)\le\gamma_{k}(x-z)=\rho_{k,B^{c}}(x).
\]
As for $\rho_{k}$, we extend $\rho_{k,B^{c}}$ to all of $\R^{n}$
by setting it equal to $0$ on $B$. Then for $x\notin U$ we have
$\rho_{k}(x)=0\le\rho_{k,B^{c}}(x)$. Hence, $\phi$ is also touching
$\rho_{k,B^{c}}$ from below at $x_{0}$.

As we will show below in Part II, $I\rho_{k,B^{c}}(x)\le C_{V}$ for
$x\in V$. So by ellipticity of $I$ we must have 
\[
I\phi(x_{0})\le I\rho_{k,B^{c}}(x_{0})+M^{+}(\phi-\rho_{k,B^{c}})(x_{0})\le C_{V},
\]
since it is easy to see that $M^{+}(\phi-\rho_{k,B^{c}})(x_{0})\le0$
as $L(\phi-\rho_{k,B^{c}})(x_{0})\le0$ for any linear operator $L$.
The upper bound for $-Iu_{k}$ can be shown to hold similarly.

\medskip{}

PART II:\medskip{}

First let us show that $\rho_{k,B^{c}}$ is $C^{2}$ on $\R^{n}-\overline{B}$.
Let $x\in\R^{n}-\overline{B}$. By the results of Section \ref{subsec: Reg Opstacles}
we need to show that $x$ has only one $\rho_{k,B^{c}}$-closest point
on $\partial B$, and $\det Q_{k,B^{c}}(x)\ne0$ where $Q_{k,B^{c}}$
is given by (\ref{eq: Q_k,B}). (Note that although $\R^{n}-\overline{B}$
is not bounded, the function $\rho_{k,B^{c}}$ can be studied as in
Section \ref{subsec: Reg Opstacles}, since $\partial B$ is compact.)
Let $y\in\partial B$ be a $\rho_{k,B^{c}}$-closest point to $x$.
Then the convex set $x-\rho_{k,B^{c}}(x)K_{k}$ intersects $\partial B$
at $y$, and has empty intersection with $B$. But $\overline{B}$
is strictly convex; so $x-\rho_{k,B^{c}}(x)K_{k}$ cannot intersect
$\partial B$ at any other point (see Figure \ref{fig: 1}). Thus
$y$ is the unique $\rho_{k,B^{c}}$-closest point on $\partial B$
to $x$. 

Next note that since the exterior data $\varphi=0$, the equations
(\ref{eq: lambda})\textendash (\ref{eq: X}) reduce to 
\begin{equation}
\lambda_{k}(y)=\frac{1}{\gamma_{k}^{\circ}(\nu)},\qquad\mu_{k}(y)=\frac{\nu}{\gamma_{k}^{\circ}(\nu)},\qquad X_{k}=\frac{1}{\gamma_{k}^{\circ}(\nu)}D\gamma_{k}^{\circ}(\nu)\otimes\nu,\label{eq: X_k}
\end{equation}
where $\nu$ is the normal to $\partial B$ at $y$. Here we also
used (\ref{eq: homog}),(\ref{eq: Euler formula}) to simplify the
last expression. Hence by (\ref{eq: D2 rho (y)}) we have 
\begin{equation}
D^{2}\rho_{k,B^{c}}(y)=\frac{1}{\gamma_{k}^{\circ}(\nu)}(I-X_{k}^{T})D^{2}d_{B^{c}}(y)(I-X_{k}),\label{eq: D2 rho _B (y)}
\end{equation}
where $d_{B^{c}}$ is the Euclidean distance to $\partial B$ on $\R^{n}-B$.
But the eigenvalues of $D^{2}d$ are minus the principal curvatures
of the boundary, and a zero eigenvalue corresponding to the normal
direction (see {[}\citealp{MR1814364}, Section 14.6{]}). So the nonzero
eigenvalues of $D^{2}d_{B^{c}}$ are $\frac{1}{r_{0}}$, where $r_{0}$
is the radius of $B$ (note that we are in the exterior of $B$, so
the principal curvatures are $\frac{-1}{r_{0}}$). Therefore $D^{2}\rho_{k,B^{c}}(y)$
is a positive semidefinite matrix. Now let us consider the matrices
$W,Q$ for $\rho_{k,B^{c}}$ given by (\ref{eq: W,Q}). The eigenvalues
of 
\[
W_{k,B^{c}}(y)=-D^{2}\gamma_{k}^{\circ}(\mu_{k}(y))D^{2}\rho_{k,B^{c}}(y)
\]
must be nonpositive, because $D^{2}\gamma_{k}^{\circ}(\mu_{k})$ is
a positive semidefinite matrix. Therefore the eigenvalues of 
\begin{equation}
Q_{k,B^{c}}(x)=I-\rho_{k,B^{c}}(x)W_{k,B^{c}}(y)=I-\gamma_{k}(x-y)W_{k,B^{c}}(y)\label{eq: Q_k,B}
\end{equation}
are greater than or equal to 1. Hence $\det Q_{k,B^{c}}(x)>0$, and
we can conclude that $\rho_{k,B^{c}}$ is $C^{2}$ on a neighborhood
of an arbitrary point $x\in\R^{n}-\overline{B}$, as desired. 

In addition, by Lemma \ref{lem: D2 rho decreas} we have 
\[
D^{2}\rho_{k,B^{c}}(x)\le D^{2}\rho_{k,B^{c}}(y)=\frac{1}{\gamma_{k}^{\circ}(\nu)}(I-X_{k}^{T})D^{2}d_{B^{c}}(y)(I-X_{k}).
\]
Let us show that $D^{2}\rho_{k,B^{c}}(y)$ is bounded independently
of $k,y,B$. Since the radius of $B$ is fixed, $D^{2}d_{B^{c}}(y)$
is bounded independently of $y,B$. So we only need to show that $X_{k},1/\gamma_{k}^{\circ}(\nu)$
are uniformly bounded. And in order to do this, by (\ref{eq: X_k}),
it suffices to show that $D\gamma_{k}^{\circ}(\nu),1/\gamma_{k}^{\circ}(\nu)$
are uniformly bounded. Note that we have $\gamma_{k}(D\gamma_{k}^{\circ}(\nu))=1$
due to (\ref{eq: g0 (Dg)=00003D1}). Thus $\gamma(D\gamma_{k}^{\circ}(\nu))\le1$
for every $k$, since $\gamma\le\gamma_{k}$ due to $K\supset K_{k}$.
So $D\gamma_{k}^{\circ}(\nu)$ is bounded independently of $k$. On
the other hand, by (\ref{eq: c<g<C}) we have $\gamma_{k}^{\circ}(\nu)\ge\gamma_{1}^{\circ}(\nu)\ge c_{1}|\nu|=c_{1}$
for some $c_{1}>0$. Hence $1/\gamma_{k}^{\circ}(\nu)$ is uniformly
bounded too. Therefore $D^{2}\rho_{k,B^{c}}(x)$ has a uniform upper
bound, independently of $k,x,B$. 

Then it follows that 
\begin{align*}
\delta\rho_{k,B^{c}}(x,h)=\rho_{k,B^{c}}(x+h) & +\rho_{k,B^{c}}(x-h)-2\rho_{k,B^{c}}(x)\\
 & =|h|^{2}\int_{0}^{1}\int_{-1}^{1}tD_{\hat{h}\hat{h}}^{2}\rho_{k,B^{c}}(x+sth)\,dsdt\le C|h|^{2}\tag{\ensuremath{\hat{h}=h/|h|}}
\end{align*}
for some constant $C$ independent of $k,x,B$, provided that the
segment $[x-h,x+h]$ does not intersect $B$. Now we truncate $\rho_{k,B^{c}}$
outside of a neighborhood of $\overline{U}$ to make it bounded. Note
that we can choose this bound uniformly, since 
\[
\rho_{k,B^{c}}(x)=\gamma_{k}(x-y)\le\gamma_{1}(x-y)\le\gamma_{1}(l),
\]
where $l$ is the diameter of the chosen neighborhood of $\overline{U}$.
Also note that we can still make sure that $\rho_{k}\le\rho_{k,B^{c}}$
after truncation, since $\rho_{k}=0$ outside $U$.

Now suppose $x\in V$ and $d(V,\partial U)>\tau$. We can make $C$
larger if necessary so that for $|h|>\tau$ we have $\delta\rho_{k,B^{c}}(x,h)\le C\tau^{2}$,
since $\rho_{k,B^{c}}$ is bounded independently of $k,B$. (It suffices
to take $C\ge\frac{4}{\tau^{2}}\sup_{k,B}\|\rho_{k,B^{c}}\|_{L^{\infty}}$.)
Then by ellipticity of $I$ we have 
\begin{align*}
I\rho_{k,B^{c}}(x) & \le I0(x)+M^{+}\rho_{k,B^{c}}(x)\\
 & =0+(1-s)\int_{\mathbb{R}^{n}}\frac{\Lambda\delta\rho_{k,B^{c}}(x,h)^{+}-\lambda\delta\rho_{k,B^{c}}(x,h)^{-}}{|h|^{n+2s}}\,dh\\
 & \le(1-s)\int_{\mathbb{R}^{n}}\frac{(\Lambda+\lambda)C\min\{\tau^{2},|h|^{2}\}}{|h|^{n+2s}}\,dh\\
 & =(1-s)(\Lambda+\lambda)C\int_{\mathbb{S}^{n-1}}\int_{0}^{\infty}\frac{\min\{\tau^{2},r^{2}\}}{r^{n+2s}}\,r^{n-1}drdS\\
 & =(1-s)\hat{C}\int_{0}^{\infty}\min\{\tau^{2},r^{2}\}\,r^{-1-2s}dr=\frac{\hat{C}}{2s}\tau^{2-2s}\le\frac{\hat{C}\tau^{2-2s}}{2s_{0}}=:C_{V}<\infty,
\end{align*}
as desired.

\medskip{}

PART III:\medskip{}

Let $x_{0}\in V$ then $B_{\tau}(x_{0})\subset U$. Thus by the estimate
(\ref{eq: I(u_e) bdd}) and Theorem 4.1 of \citep{kriventsov2013c}
we have 
\[
\|u_{k}\|_{C^{1,\alpha}(B_{\tau/2}(x_{0}))}\le\frac{C}{\tau^{1+\alpha}}(\|u_{k}\|_{L^{\infty}(\mathbb{R}^{n})}+C_{V}\tau^{2s}),
\]
where $C,\alpha$ depend only on $n,s_{0},\lambda,\Lambda$. For simplicity
we are assuming that $\tau\le1$. %
\begin{comment}
$v(y):=\tilde{u}(x_{0}+ry)$ , , , , $r^{2s}I[0]=0$
\end{comment}
(Note that by considering the scaled operator $I_{\tau}v(\cdot)=\tau^{2s}Iv(\frac{\cdot}{\tau})$,
which has the same ellipticity constants $\lambda,\Lambda$ as $I$,
and using the translation invariance of $I$, we have obtained the
estimate on the domain $B_{\tau/2}(x_{0})$ instead of $B_{1/2}(0)$.)
Then we can cover $V\subset\subset U$ with finitely many open balls
contained in $U$ and obtain 
\begin{equation}
\|u_{k}\|_{C^{1,\alpha}(\overline{V})}\le C(\|u_{k}\|_{L^{\infty}(\mathbb{R}^{n})}+C_{V}),\label{eq: u_k bdd C1,a}
\end{equation}
where $C$ depends only on $n,\lambda,\Lambda,s_{0}$, and $d(V,\partial U)$.
In particular $C$ does not depend on $k$. 

Therefore $u_{k}$ is bounded in $C^{1,\alpha}(\overline{V})$ independently
of $k$, because $\|u_{k}\|_{L^{\infty}}$ is uniformly bounded by
(\ref{eq: u_k bdd}). Hence there is a subsequence of $u_{k}$ that
is convergent in $C^{1}$ norm to a function in $C^{1,\alpha}(\overline{V})$.
But the limit must be $u$, since $u_{k}$ converges uniformly to
$u$ on $U$. Thus $u\in C_{\textrm{loc}}^{1,\alpha}(U)$. Also note
that $u_{k}$ converges uniformly to $u$ on $\R^{n}$, because $u_{k},u$
are zero outside of $U$. In addition, if we let $k\to\infty$ in
(\ref{eq: u_k bdd C1,a}), we obtain the same estimate for $\|u\|_{C^{1,\alpha}(\overline{V})}$.

Finally, let us show that due to the stability of viscosity solutions,
$u$ must satisfy the double obstacle problem (\ref{eq: dbl obstcl})
for zero exterior data. Suppose $\phi$ is a bounded $C^{2}$ function
and $u-\phi$ has a maximum over $\R^{n}$ at $x_{0}\in U$. Let us
first consider the case where $u-\phi$ has a strict maximum at $x_{0}$.
We must show that at $x_{0}$ we have 
\begin{equation}
\max\{\min\{-I\phi(x_{0}),\;u+\bar{\rho}\},u-\rho\}\le0.\label{eq: 3 in Reg u}
\end{equation}
Now we know that $u_{k}-\phi$ takes its global maximum at a point
$x_{k}$ where $x_{k}\to x_{0}$; because $u_{k}$ uniformly converges
to $u$ on $\R^{n}$.

We also know that $-\bar{\rho}\leq u\leq\rho$. If $-\bar{\rho}(x_{0})=u(x_{0})$
then (\ref{eq: 3 in Reg u}) holds trivially. So suppose $-\bar{\rho}(x_{0})<u(x_{0})$.
Then for large $k$ we have $-\bar{\rho}_{k}(x_{k})<u_{k}(x_{k})$,
since $u_{k}+\bar{\rho}_{k}$ locally uniformly converges to $u+\bar{\rho}$.
Hence since $u_{k}$ is a viscosity solution of the equation (\ref{eq: dbl obstcl k}),
at $x_{k}$ we have 
\[
\max\{\min\{-I\phi(x_{k}),\;u_{k}+\bar{\rho}_{k}\},u_{k}-\rho_{k}\}\le0.
\]
But $u_{k}+\bar{\rho}_{k}>0$ at $x_{k}$, so we must have $-I\phi(x_{k})\le0$.
Thus by letting $k\to\infty$ and using the continuity of $I\phi$
we see that (\ref{eq: 3 in Reg u}) holds in this case too. 

Now if the maximum of $u-\phi$ at $x_{0}$ is not strict, we can
approximate $\phi$ with $\phi_{\epsilon}=\phi+\epsilon\tilde{\phi}$,
where $\tilde{\phi}$ is a bounded $C^{2}$ functions which vanishes
at $x_{0}$ and is positive elsewhere. Then, as we have shown, when
$-\bar{\rho}(x_{0})<u(x_{0})$ we have $-I\phi_{\epsilon}(x_{0})\le0$.
Hence by the ellipticity of $I$ we get 
\[
-I\phi(x_{0})\le M^{+}(\epsilon\tilde{\phi})(x_{0})-I\phi_{\epsilon}(x_{0})\le\epsilon M^{+}\tilde{\phi}(x_{0})\underset{\epsilon\to0}{\longrightarrow}0,
\]
as desired. Similarly, we can show that when $u-\phi$ has a minimum
at $x_{0}\in U$ we have 
\[
\max\{\min\{-I\phi(x_{0}),\;u+\bar{\rho}\},u-\rho\}\ge0.
\]
Therefore $u$ is a viscosity solution of equation (\ref{eq: dbl obstcl})
as desired. At the end note that by Lemma \ref{lem: g(Du)<1} we also
have $u\in C^{0,1}(\overline{U})$, since $u$ and its derivative
are bounded, and $\partial U$ is smooth enough.
\end{proof}

\appendix

\section{\label{sec: App A}Proof of Theorem \ref{thm: Reg dbl obstcl} for
nonzero exterior data}

We do not use the next two propositions directly in the proof of Theorem
\ref{thm: Reg dbl obstcl}, however, they enhance our understanding
of double obstacle problems and equations with gradient constraints.
First let us introduce the following terminology for the solutions
of the double obstacle problem (\ref{eq: dbl obstcl}). (The notation
is motivated by the physical properties of the elastic-plastic torsion
problem, in which $E$ stands for the \textit{elastic} region, and
$P$ stands for the \textit{plastic} region.)
\begin{defn}
\label{def: plastic}Let 
\begin{eqnarray*}
P^{+}:=\{x\in U:u(x)=\rho(x)\}, &  & P^{-}:=\{x\in U:u(x)=-\bar{\rho}(x)\}.
\end{eqnarray*}
Then $P:=P^{+}\cup P^{-}$ is called the \textbf{coincidence} set;
and 
\[
E:=\{x\in U:-\bar{\rho}(x)<u(x)<\rho(x)\}
\]
is called the \textbf{non-coincidence} set. We also define the \textbf{free
boundary} to be $\partial E\cap U$.
\end{defn}

\begin{prop}
\label{prop: segment is plastic}Suppose Assumption \ref{assu: all}
holds. Let $u\in C^{1}(U)$ be a viscosity solution of the double
obstacle problem (\ref{eq: dbl obstcl}). If $x\in P^{+}$, and $y$
is a $\rho$-closest point on $\partial U$ to $x$ such that $[x,y[\subset U$,
then we have $[x,y[\subset P^{+}$. Similarly, if $x\in P^{-}$, and
$y$ is a $\bar{\rho}$-closest point on $\partial U$ to $x$ such
that $[x,y[\subset U$, then we have $[x,y[\subset P^{-}$.%
\begin{comment}
Is this still true when the constraint depends on $x$ (and may be
$u$) ?
\end{comment}
\end{prop}
\begin{rem*}
If the strict Lipschitz property (\ref{eq: phi strct Lip}) for $\varphi$
holds, then by Lemma \ref{lem: segment to the closest pt} we automatically
have $[x,y[\subset U$. 
\end{rem*}
\begin{proof}
Suppose $x\in P^{-}$; the other case is similar. We have 
\[
u(x)=-\bar{\rho}(x)=-\gamma(y-x)+\varphi(y).
\]
Let $\tilde{v}:=u-(-\bar{\rho})\ge0$, and $\xi:=\frac{y-x}{\gamma(y-x)}=-\frac{x-y}{\bar{\gamma}(x-y)}$.
Then $\bar{\rho}$ varies linearly along the segment $]x,y[$, since
$y$ is a $\bar{\rho}$-closest point to the points of the segment.
So we have $D_{\xi}(-\bar{\rho})=D_{-\xi}\bar{\rho}=1$ along the
segment. Note that we do not assume the differentiability of $\bar{\rho}$;
and $D_{-\xi}\bar{\rho}$ is just the derivative of the restriction
of $\bar{\rho}$ to the segment $]x,y[$. Now by Lemma \ref{lem: g(Du)<1}
we get 
\[
D_{\xi}u=\langle Du,\xi\rangle\le\gamma^{\circ}(Du)\gamma(\xi)\le1.
\]
So we have $D_{\xi}\tilde{v}\le0$ along $]x,y[$. Thus as $\tilde{v}(x)=\tilde{v}(y)=0$,
and $\tilde{v}$ is continuous on the closed segment $[x,y]$, we
must have $\tilde{v}\equiv0$ on $[x,y]$. Therefore $u=-\bar{\rho}$
along the segment as desired.
\end{proof}
\begin{prop}
\label{prop: ridge is elastic}Suppose Assumption \ref{assu: all}
holds. Let $u\in C^{1}(U)$ be a viscosity solution of the double
obstacle problem (\ref{eq: dbl obstcl}). Suppose Assumption \ref{assu: K,U}
holds too. Then we have 
\begin{eqnarray*}
R_{\rho,0}\cap P^{+}=\emptyset, & \hspace{2cm} & R_{\bar{\rho},0}\cap P^{-}=\emptyset.
\end{eqnarray*}
\end{prop}
\begin{proof}
Note that due to Assumption \ref{assu: K,U}, the strict Lipschitz
property (\ref{eq: phi strct Lip}) for $\varphi$ holds, and $\gamma$
is strictly convex. Let us show that $R_{\bar{\rho},0}\cap P^{-}=\emptyset$;
the other case is similar. Suppose to the contrary that $x\in R_{\bar{\rho},0}\cap P^{-}$.
Then there are at least two distinct points $y,z\in\partial U$ such
that 
\[
\bar{\rho}(x)=\gamma(y-x)-\varphi(y)=\gamma(z-x)-\varphi(z).
\]
Now by Lemma \ref{lem: segment to the closest pt} we know that $[x,y[,[x,z[\subset U$;
so by Proposition \ref{prop: segment is plastic} we get $[x,y[,[x,z[\subset P^{-}$.%
\begin{comment}
If we try to weaken the assumptions and use the closest points on
these segments to $x$, then $y,z$ can become the same point!? Also,
the fact that $R_{\rho,0}\subset R_{\rho}$ is only true with these
assumptions!?
\end{comment}
{} In other words, $u=-\bar{\rho}$ on both of these segments. Therefore
by Lemma \ref{lem: segment to the closest pt}, $u$ varies linearly
on both of these segments. Hence we get 
\[
\big\langle Du(x),\frac{y-x}{\gamma(y-x)}\big\rangle=1=\big\langle Du(x),\frac{z-x}{\gamma(z-x)}\big\rangle.
\]
However, since $\gamma$ is strictly convex, and by Lemma \ref{lem: g(Du)<1}
we know that $\gamma^{\circ}(Du(x))\le1$, this equality implies that
$z,y$ must be on the same ray emanating from $x$. But this contradicts
the fact that $[x,y[,[x,z[\subset U$.
\end{proof}
\begin{rem*}
Since here we do not have $C^{1,1}$ regularity for $u$, we cannot
imitate the proof given in the local case to show that $P^{+},P^{-}$
do not intersect $R_{\rho},R_{\bar{\rho}}$. In addition, merely knowing
that the solution does not touch the obstacles at their singularities
is not enough to obtain uniform bounds for $Iu_{k}$ in the proof
of Theorem \ref{thm: Reg dbl obstcl}, because we need to have a uniform
positive distance from the ridges too, due to the nonlocal nature
of the operator. As we have seen in the proof of Theorem \ref{thm: Reg dbl obstcl},
we overcome these limitations by using some suitable barriers. 
\end{rem*}
\begin{comment}
As we will see in the following proof, in order to show that $u\in W^{2,p}(U)\cap W_{\mathrm{loc}}^{2,\infty}(U)$
for every $p<\infty$, we only need $\partial U,\varphi$ to be $C^{2}$.
But we need their $C^{2,\alpha}$ regularity to be able to apply the
result of \citep{indrei2016nontransversal}, and conclude the optimal
regularity of $u$ up to the boundary. (SINCE we used existence of
$C^{2,\alpha}$ solutions of fully nonlinear equations to show the
existence of $u_{k}$, this does not work ?? But we can approximate
$\varphi$ with $\varphi_{k}$, though probably not $\partial U$
??)
\end{comment}

Next let us review some well-known facts from convex analysis which
are needed in the following proof. Consider a compact convex set $K$.
Let $x\in\partial K$ and $\mathrm{v}\in\R^{n}-\{0\}$. We say the
hyperplane 
\begin{equation}
\Gamma_{x,\mathrm{v}}:=\{y\in\R^{n}:\langle y-x,\mathrm{v}\rangle=0\}\label{eq: hyperplane}
\end{equation}
is a \textit{supporting hyperplane} of $K$ at $x$ if $K\subset\{y:\langle y-x,\mathrm{v}\rangle\le0\}$.
In this case we say $\mathrm{v}$ is an \textit{outer normal vector}
of $K$ at $x$ (see Figure \ref{fig: 3}). The \textit{normal cone}
of $K$ at $x$ is the closed convex cone 
\begin{equation}
N(K,x):=\{0\}\cup\{\mathrm{v}\in\mathbb{R}^{n}-\{0\}:\mathrm{v}\textrm{ is an outer normal vector of }K\textrm{ at }x\}.\label{eq: normal cone}
\end{equation}
It is easy to see that when $\partial K$ is $C^{1}$ we have 
\[
N(K,x)=\{tD\gamma(x):t\ge0\}.
\]
For more details see {[}\citealp{MR3155183}, Sections 1.3 and 2.2{]}.
\begin{proof}[\textbf{Proof of Theorem \ref{thm: Reg dbl obstcl} for nonzero exterior
data}]
 As before we approximate $K^{\circ}$ by a sequence $K_{k}^{\circ}$
of compact convex sets, that have smooth boundaries with positive
curvature, and 
\begin{eqnarray*}
K_{k+1}^{\circ}\subset\mathrm{int}(K_{k}^{\circ}), & \qquad & K^{\circ}={\textstyle \bigcap}K_{k}^{\circ}.
\end{eqnarray*}
Then $K_{k}$'s are strictly convex compact sets with $0$ in their
interior, which have smooth boundaries with positive curvature. Furthermore
we have $K=(K^{\circ})^{\circ}\supset K_{k+1}\supset K_{k}$. To simplify
the notation we use $\gamma_{k},\gamma_{k}^{\circ},\rho_{k},\bar{\rho}_{k}$
instead of $\gamma_{K_{k}},\gamma_{K_{k}^{\circ}},\rho_{K_{k},\varphi},\bar{\rho}_{K_{k},\varphi}$,
respectively. Note that $K_{k},U,\varphi$ satisfy the Assumption
\ref{assu: K,U}. In particular we have $\gamma_{k}^{\circ}(D\varphi)<1$,
since $D\varphi\in K^{\circ}\subset\mathrm{int}(K_{k}^{\circ})$.
Hence as we have shown in \citep{SAFDARI202176}, $\rho_{k},\bar{\rho}_{k}$
satisfy the assumptions of Theorem 1 of \citep{Safd-Nonlocal-dbl-obst}.
Thus there are viscosity solutions $u_{k}\in C_{\mathrm{loc}}^{1,\alpha}(U)$
of the double obstacle problem 
\begin{equation}
\begin{cases}
\max\{\min\{-Iu_{k},\;u_{k}+\bar{\rho}_{k}\},u_{k}-\rho_{k}\}=0 & \textrm{in }U,\\
u_{k}=\varphi & \textrm{in }\mathbb{R}^{n}-U.
\end{cases}\label{eq: dbl obstcl k-1}
\end{equation}
And $\alpha>0$ depends only on $n,\lambda,\Lambda,s_{0}$. 

As before, we can easily see that $u_{k}$'s and their derivatives
are uniformly bounded. Hence a subsequence of them converges uniformly
to a continuous function $u$, which we extend to all of $\R^{n}$
by setting it equal to $\varphi$ in $\R^{n}-U$. In Part I we consider
the barrier 
\begin{equation}
\rho_{k,B^{c}}:=\rho_{K_{k},\varphi}(\cdot\,;\R^{n}-\overline{B})=\underset{z\in\partial B}{\min}[\gamma_{k}(\cdot-z)+\varphi(z)],\label{eq: barrier 2}
\end{equation}
where $B\subset\R^{n}-\overline{U}$ is an open ball such that $\partial B\cap\partial U=\{y_{0}\}$,
in which $y_{0}\in\partial U$ is a $\rho_{k}$-closest point to $x_{0}$. 

Let $z\in\partial B$, and let $y$ be a point on $\partial U\cap[z,x_{0}[$.
Then by the Lipschitz property (\ref{eq: phi Lip}) for $\varphi$
with respect to $\gamma_{k}$ (note that $\gamma_{k}^{\circ}(D\varphi)<1$)
we have 
\begin{align*}
\gamma_{k}(x_{0}-y_{0}) & +\varphi(y_{0})\le\gamma_{k}(x_{0}-y)+\varphi(y)\\
 & =\gamma_{k}(x_{0}-z)+\varphi(z)-\gamma_{k}(y-z)+\varphi(y)-\varphi(z)\le\gamma_{k}(x_{0}-z)+\varphi(z).
\end{align*}
Hence $y_{0}$ is also a $\rho_{k,B^{c}}$-closest point to $x_{0}$
on $\partial B$.%
\begin{comment}
$y$ can also be equal to $y_{0}$
\end{comment}
{} Consequently we have 
\[
\rho_{k,B^{c}}(x_{0})=\gamma_{k}(x_{0}-y_{0})+\varphi(y_{0})=\rho_{k}(x_{0}).
\]
Now consider $x\in U\subset\R^{n}-\overline{B}$, and let $z\in\partial B$
be a $\rho_{k,B^{c}}$-closest point to $x$. Let $y$ be a point
on $\partial U\cap[z,x[$. Then similarly to the above we can show
that 
\[
\rho_{k}(x)\le\gamma_{k}(x-y)+\varphi(y)\le\gamma_{k}(x-z)+\varphi(z)=\rho_{k,B^{c}}(x).
\]
As for $\rho_{k}$, we extend $\rho_{k,B^{c}}$ to all of $\R^{n}$
by setting it equal to $\varphi$ on $B$. Then for $x\notin U$ we
either have $\rho_{k}(x)=\varphi(x)=\rho_{k,B^{c}}(x)$ when $x\in B$,
or 
\[
\rho_{k}(x)=\varphi(x)\le\gamma_{k}(x-z)+\varphi(z)=\rho_{k,B^{c}}(x),
\]
when $x\notin B$ and $z\in\partial B$ is a $\rho_{k,B^{c}}$-closest
point to $x$. Hence, $\phi$ is also touching $\rho_{k,B^{c}}$ from
below at $x_{0}$. 

The rest of the proof in Part I, and Part III of the proof, go as
before. So we only need to prove the properties of the new barrier
$\rho_{k,B^{c}}$, similarly to the Part II of the proof in the case
of zero exterior data. First let us show that $\rho_{k,B^{c}}$ is
$C^{2}$ on $\R^{n}-\overline{B}$. Let $x\in\R^{n}-\overline{B}$.
By the results of Section \ref{subsec: Reg Opstacles} we need to
show that $x$ has only one $\rho_{k,B^{c}}$-closest point on $\partial B$,
and $\det Q_{k,B^{c}}(x)\ne0$ where $Q_{k,B^{c}}$ is given by (\ref{eq: Q_k,B-1}).

Suppose to the contrary that $y,\tilde{y}\in\partial B$ are two $\rho_{k,B^{c}}$-closest
points to $x$. Let $z=\frac{y+\tilde{y}}{2}\in B$. We assume that
the radius of $B$ is small enough so that $\varphi$ is convex on
a neighborhood of it. Then we have 
\[
\gamma_{k}(x-z)+\varphi(z)\le\frac{1}{2}\big(\gamma_{k}(x-y)+\gamma_{k}(x-\tilde{y})+\varphi(y)+\varphi(\tilde{y})\big)=\rho_{k,B^{c}}(x).
\]
Let $\tilde{z}$ be the point on $\partial B\,\cap\,]z,x[$. Then
by the strict Lipschitz property (\ref{eq: phi strct Lip}) for $\varphi$
with respect to $\gamma_{k}$ (note that $\gamma_{k}^{\circ}(D\varphi)<1$)
we have 
\begin{align*}
\gamma_{k}(x-\tilde{z})+\varphi(\tilde{z})=\gamma_{k}(x-z)+\varphi(z) & -\gamma_{k}(\tilde{z}-z)+\varphi(\tilde{z})-\varphi(z)\\
 & <\gamma_{k}(x-z)+\varphi(z)\le\rho_{k,B^{c}}(x),
\end{align*}
which is a contradiction. So $x$ must have a unique $\rho_{k,B^{c}}$-closest
point $y$ on $\partial B$.

Next note that by (\ref{eq: D2 rho (y)}) at $y\in\partial B$ we
have 
\begin{equation}
D^{2}\rho_{k,B^{c}}(y)=(I-X_{k}^{T})\big(D^{2}\varphi(y)+\lambda_{k}(y)D^{2}d_{B^{c}}(y)\big)(I-X_{k}),\label{eq: D2 rho _B (y)-2}
\end{equation}
where $d_{B^{c}}$ is the Euclidean distance to $\partial B$ on $\R^{n}-B$,
and $\lambda_{k},X_{k}$ are given by (\ref{eq: lambda}),(\ref{eq: X})
(using $\gamma_{k}^{\circ}$ instead of $\gamma^{\circ}$). But the
eigenvalues of $D^{2}d$ are minus the principal curvatures of the
boundary, and a zero eigenvalue corresponding to the normal direction
(see {[}\citealp{MR1814364}, Section 14.6{]}). So the nonzero eigenvalues
of $D^{2}d_{B^{c}}$ are $\frac{1}{r_{0}}$, where $r_{0}$ is the
radius of $B$ (note that we are in the exterior of $B$, so the curvatures
are $\frac{-1}{r_{0}}$). Hence $D^{2}d_{B^{c}}$ is a positive semidefinite
matrix. Therefore $D^{2}\rho_{k,B^{c}}(y)$ is also a positive semidefinite
matrix, since $\varphi$ is convex on a neighborhood of $\partial U$.
(Although the convexity of $\varphi$ is not really needed here. Because
by using $\gamma^{\circ}(D\varphi)<1$ we can easily show that $\lambda_{k}$
has a uniform positive lower bound independently of $k,B$. Then by
decreasing the radius $r_{0}$ and using the boundedness of $D^{2}\varphi$
we can get the desired.)

Now let us consider the matrices $W,Q$ for $\rho_{k,B^{c}}$ given
by (\ref{eq: W,Q}). The eigenvalues of 
\[
W_{k,B^{c}}(y_{0}):=-D^{2}\gamma_{k}^{\circ}(\mu_{k}(y))D^{2}\rho_{k,B^{c}}(y)
\]
must be nonpositive, because $D^{2}\gamma_{k}^{\circ}(\mu_{k})$ is
a positive semidefinite matrix. Therefore the eigenvalues of 
\begin{equation}
Q_{k,B^{c}}(x):=I-\big(\rho_{k,B^{c}}(x)-\varphi(y)\big)W_{k,B^{c}}(y)=I-\gamma_{k}(x-y)W_{k,B^{c}}(y)\label{eq: Q_k,B-1}
\end{equation}
are greater than or equal to 1. Hence $\det Q_{k,B^{c}}(x)>0$, and
we can conclude that $\rho_{k,B^{c}}$ is $C^{2}$ on a neighborhood
of an arbitrary point $x\in\R^{n}-\overline{B}$, as desired. 

In addition, by Lemma \ref{lem: D2 rho decreas} we have 
\[
D^{2}\rho_{k,B^{c}}(x)\le D^{2}\rho_{k,B^{c}}(y)=(I-X_{k}^{T})\big(D^{2}\varphi(y)+\lambda_{k}(y)D^{2}d_{B^{c}}(y)\big)(I-X_{k}).
\]
Let us show that $D^{2}\rho_{k,B^{c}}(y)$ is bounded independently
of $k,y,B$. Since the radius of $B$ is fixed, $D^{2}d_{B^{c}}(y)$
is bounded independently of $y,B$. So we only need to show that $X_{k},\lambda_{k}$
are uniformly bounded. Note that $\gamma_{k}^{\circ}\ge\gamma_{1}^{\circ}$,
since $K_{k}^{\circ}\subset K_{1}^{\circ}$. Thus by (\ref{eq: lambda})
we have 
\[
\gamma_{1}^{\circ}(D\varphi+\lambda_{k}\nu)\le\gamma_{k}^{\circ}(D\varphi+\lambda_{k}\nu)=1.
\]
Hence by (\ref{eq: c<g<C}) applied to $\gamma_{1}^{\circ}$, we have
$|D\varphi+\lambda_{k}\nu|\le C$ for some $C>0$. Therefore we get
$|\lambda_{k}|=|\lambda_{k}\nu|\le C+|D\varphi|$. Thus $\lambda_{k}$
is bounded independently of $k,y,B$.

Hence we only need to show that the entries of $X_{k}=\frac{1}{\langle D\gamma_{k}^{\circ}(\mu_{k}),\nu\rangle}D\gamma_{k}^{\circ}(\mu_{k})\otimes\nu$
are bounded. Note that by (\ref{eq: g0 (Dg)=00003D1}) we have $\gamma_{k}(D\gamma_{k}^{\circ}(\nu))=1$.
Thus $\gamma(D\gamma_{k}^{\circ}(\nu))\le1$ for every $k$, since
$\gamma\le\gamma_{k}$ due to $K\supset K_{k}$. So $D\gamma_{k}^{\circ}(\nu)$
is bounded independently of $k$. Therefore it only remains to show
that $\langle D\gamma_{k}^{\circ}(\mu_{k}),\nu\rangle$ has a positive
lower bound independently of $k,y,B$.%
\begin{comment}
The deep reason that this works is that $\langle\mathrm{v},\nu\rangle$,
for $\mathrm{v}\in N(K^{\circ},\mu)$ with $\gamma(\mathrm{v})=1$
($\mu=D\rho(y)$ ??), is a positive?? continuous function on the compact
set $\partial U$ (Is it ??)
\end{comment}
{} Note that for every $k,B$, $\langle D\gamma_{k}^{\circ}(\mu_{k}),\nu\rangle$
is a continuous positive function on the compact set $\partial B$,
as explained in Section \ref{subsec: Reg Opstacles}. Hence there
is $c_{k,B}>0$ such that $\langle D\gamma_{k}^{\circ}(\mu_{k}),\nu\rangle\ge c_{k,B}$.
Suppose to the contrary that there is a sequence of balls $B_{j}$,
points $y_{j}\in\partial B_{j}$, and $k_{j}$ (which we simply denote
by $j$) such that 
\begin{equation}
\langle D\gamma_{j}^{\circ}(\mu_{j}(y_{j})),\nu_{j}\rangle\to0,\label{eq: 2 in reg Nonconv dom}
\end{equation}
where $\nu_{j}$ is the unit normal to $\partial B_{j}$ at $y_{j}$.
By passing to another subsequence, we can assume that $y_{j}\to y$,
since $y_{j}$'s belong to a compact neighborhood of $\partial U$.
Now remember that 
\[
\mu_{j}(y_{j})=D\varphi(y_{j})+\lambda_{j}(y_{j})\nu_{j},
\]
where $\lambda_{j}>0$. As we have shown in the last paragraph, $\lambda_{j}$
is bounded independently of $j,B_{j}$. Hence by passing to another
subsequence, we can assume that $\lambda_{j}\to\lambda^{*}\ge0$.
Also, $|\nu_{j}|=1$. Thus by passing to yet another subsequence we
can assume that $\nu_{j}\to\nu^{*}$, where $|\nu^{*}|=1$. Therefore
we have 
\[
\mu_{j}(y_{j})\to\mu^{*}:=D\varphi(y)+\lambda^{*}\nu^{*}.
\]
On the other hand we have $\gamma_{j}^{\circ}(\mu_{j}(y_{j}))=1$.
Hence $\gamma^{\circ}(\mu_{j}(y_{j}))\ge1$, since $\gamma_{j}^{\circ}\le\gamma^{\circ}$
due to $K^{\circ}\subset K_{j}^{\circ}$. Thus we get $\gamma^{\circ}(\mu^{*})\ge1$.
However we cannot have $\gamma^{\circ}(\mu^{*})>1$. Because then
$\mu^{*}$ will have a positive distance from $K^{\circ}$, and therefore
it will have a positive distance from $K_{j}^{\circ}$ for large enough
$j$. But this contradicts the facts that $\mu_{j}(y_{j})\to\mu^{*}$
and $\mu_{j}(y_{j})\in K_{j}^{\circ}$. Thus we must have $\gamma^{\circ}(\mu^{*})=1$,
i.e. $\mu^{*}\in\partial K^{\circ}$.

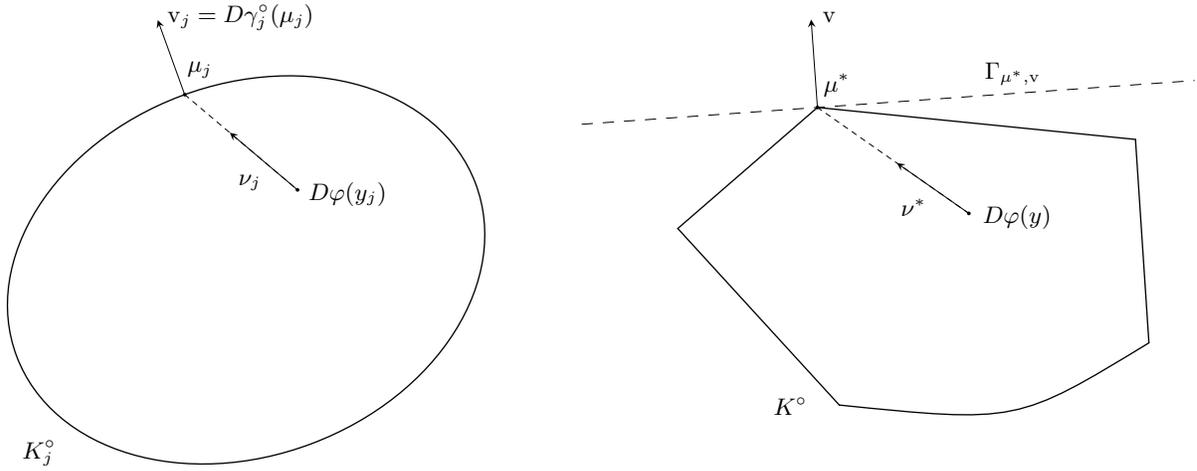
\begin{figure}

\begin{tikzpicture}[line cap=round,line join=round,>=triangle 45,x=1.0cm,y=1.0cm] 

\clip(-2,0) rectangle (16,8);

\draw [rotate around={20.650644659511226:(1.7498111476538434,3.6381515976564547)},line width=0.5pt] (1.7498111476538434,3.6381515976564547) ellipse (3.2597661359843753cm and 2.470521252956651cm);

\draw [line width=0.3pt,dash pattern=on 2pt off 2pt] (0.9335153724987351,5.970231839275003)-- (2.432888379696236,4.70331909999479);

\draw [->,line width=0.3pt, >=stealth] (2.432888379696236,4.70331909999479) -- (1.532548864499495,5.464071491830982);

\draw [->,line width=0.3pt, >=stealth] (0.9335153724987351,5.970231839275003) -- (0.5750085988536315,6.962224483050994);

\draw [line width=0.5pt] (9.345119644527374,5.7994334429615355)-- (7.488436165624511,4.188816449214472);

\draw [line width=0.5pt] (7.488436165624511,4.188816449214472)-- (9.635925490620593,1.84); 

\draw [line width=0.5pt] (9.635925490620593,1.84).. controls (12,1.6) .. (13.751946696863083,2.6676781773422444);

\draw [line width=0.5pt] (13.751946696863083,2.6676781773422444)-- (13.57298925311342,5.374409514056062); 

\draw [line width=0.5pt] (13.57298925311342,5.374409514056062)-- (9.345119644527374,5.7994334429615355);

\draw [line width=0.3pt,dash pattern=on 2pt off 2pt] (11.35839088671121,4.390143573432856)-- (9.345119644527374,5.7994334429615355);

\draw [->,line width=0.3pt, >=stealth] (11.35839088671121,4.390143573432856) -- (10.40204951365297,5.059582534573621); 

\draw [->,line width=0.3pt, >=stealth] (9.345119644527374,5.7994334429615355) -- (9.261545609882855,6.969469927984784); 

\draw [line width=0.3pt,dash pattern=on 4pt off 4pt] (6.214954102393178,5.575850189951949)-- (14.355928069518242,6.157348330460884);

\begin{scriptsize} 

\draw[color=black] (-1.,1.2) node {$K^\circ _j$};

\draw [fill=black] (2.432888379696236,4.70331909999479) circle (0.5pt);

\draw[color=black] (3.1,4.65) node {$D\varphi (y_j)$};

\draw [fill=black] (0.9335153724987351,5.970231839275003) circle (0.5pt); 

\draw[color=black] (1.1354469125111017,6.3) node  {$\mu_j$}; 

\draw[color=black] (1.8,4.8) node {$\nu_j$}; 

\draw[color=black] (1.68,7.) node {$\mathrm{v}_j = D\gamma_{j}^{\circ}(\mu_{j})$}; 

\draw [fill=black] (9.345119644527374,5.7994334429615355) circle (0.5pt); 

\draw[color=black] (9.6,6.1) node {$\mu ^*$}; 

\draw[color=black] (8.987204757028033,1.8288151597656506) node {$K^\circ$}; 

\draw [fill=black] (11.35839088671121,4.390143573432856) circle (0.5pt); 

\draw[color=black] (12,4.34) node {$D\varphi (y)$}; 

\draw[color=black] (10.620191431243805,4.468437455073338) node {$\nu ^*$}; 

\draw[color=black] (9.5,7.) node {$\mathrm{v}$}; 

\draw[color=black] (11.940002578897648,6.25) node {$\Gamma_{\mu ^{*} ,\mathrm{v}}$}; 

\end{scriptsize} 
\end{tikzpicture}

\caption{The relative situation of $\mathrm{v},\,\nu^{*}$ does not allow $\langle\mathrm{v},\nu^{*}\rangle$  to be zero.}
\label{fig: 3}

\end{figure}

Now note that $\mathrm{v}_{j}:=D\gamma_{j}^{\circ}(\mu_{j}(y_{j}))$
belongs to the normal cone $N(K_{j}^{\circ},\mu_{j}(y_{j}))$. In
addition we have $\gamma_{j}(\mathrm{v}_{j})=1$ due to (\ref{eq: g0 (Dg)=00003D1}).
Hence we have $\mathrm{v}_{j}\in K_{j}\subset K$. Thus by passing
to yet another subsequence we can assume that $\mathrm{v}_{j}\to\mathrm{v}\in K$.
We also have $\gamma_{1}(\mathrm{v}_{j})\ge1$, since $\gamma_{j}\le\gamma_{1}$
due to $K_{j}\supset K_{1}$. So we get $\gamma_{1}(\mathrm{v})\ge1$.
In particular $\mathrm{v}\ne0$. We claim that $\mathrm{v}\in N(K^{\circ},\mu^{*})$.
To see this note that we have 
\[
K^{\circ}\subset K_{j}^{\circ}\subset\{z:\langle z-\mu_{j}(y_{j}),\mathrm{v}_{j}\rangle\le0\}.
\]
Hence for every $z\in K^{\circ}$ we have $\langle z-\mu_{j}(y_{j}),\mathrm{v}_{j}\rangle\le0$.
But as $j\to\infty$ we have $z-\mu_{j}(y_{j})\to z-\mu^{*}$. So
we get $\langle z-\mu^{*},\mathrm{v}\rangle\le0$. Therefore 
\[
K^{\circ}\subset\{z:\langle z-\mu^{*},\mathrm{v}\rangle\le0\},
\]
as desired. 

On the other hand, by (\ref{eq: 2 in reg Nonconv dom}) we obtain
\begin{equation}
\langle\mathrm{v},\nu^{*}\rangle=\lim\langle\mathrm{v}_{j},\nu_{j}\rangle=0.\label{eq: 3 in reg Nonconv dom}
\end{equation}
Now note that $D\varphi=\mu^{*}-\lambda^{*}\nu^{*}$ belongs to the
ray passing through $\mu^{*}\in\partial K^{\circ}$ in the direction
$-\nu^{*}$. However, we know that $D\varphi$ is in the interior
of $K^{\circ}$, since $\gamma^{\circ}(D\varphi)<1$. Hence we must
have $\lambda^{*}>0$ (since $\gamma^{\circ}(\mu^{*})=1$). And thus
the ray $t\mapsto\mu^{*}-t\nu^{*}$ for $t>0$ passes through the
interior of $K^{\circ}$. Therefore this ray and $K^{\circ}$ must
lie on the same side of the supporting hyperplane $\Gamma_{\mu^{*},\mathrm{v}}$.
In addition, the ray cannot lie on the hyperplane, since it intersects
the interior of $K^{\circ}$. Hence we must have $\langle\mathrm{v},\nu^{*}\rangle=-\langle\mathrm{v},-\nu^{*}\rangle>0$,
which contradicts (\ref{eq: 3 in reg Nonconv dom}). See Figure \ref{fig: 3}
for a geometric representation of this argument.

Thus $\langle D\gamma_{k}^{\circ}(\mu_{k}),\nu\rangle$ must have
a positive lower bound independently of $k,y,B$, as desired. Therefore
$D^{2}\rho_{k,B^{c}}(x)$ has a uniform upper bound, independently
of $k,x,B$. Then as before it follows that 
\[
\delta\rho_{k,B^{c}}(x,h)=\rho_{k,B^{c}}(x+h)+\rho_{k,B^{c}}(x-h)-2\rho_{k,B^{c}}(x)\le C|h|^{2}
\]
for some constant $C$ independent of $k,x,B$, provided that the
segment $[x-h,x+h]$ does not intersect $B$. Now we truncate $\rho_{k,B^{c}}$
outside of a neighborhood of $\overline{U}$ to make it bounded. Note
that we can choose this bound uniformly, since 
\[
\rho_{k,B^{c}}(x)=\gamma_{k}(x-y)+\varphi(y)\le\gamma_{1}(x-y)+\varphi(y)\le\gamma_{1}(l)+\|\varphi\|_{L^{\infty}},
\]
where $l$ is the diameter of the chosen neighborhood of $\overline{U}$.
Also note that we can still make sure that $\rho_{k}\le\rho_{k,B^{c}}$
after truncation, since $\rho_{k}=\varphi$ outside $U$. And finally
we can show that $I\rho_{k,B^{c}}(x)\le C_{V}$ for $x\in V$, similarly
to the case of zero exterior data.
\end{proof}
\begin{acknowledgement*}
This research was in part supported by Iran National Science Foundation
Grant No 97012372.
\end{acknowledgement*}
\bibliographystyle{plainnat}
\bibliography{/Volumes/A/Dropbox/Bibliography-Jan-2021}

\end{document}